\documentclass[11pt]{amsart}
\usepackage[dvipsnames,usenames]{color}
\usepackage{hyperref}
\usepackage{graphicx}
\usepackage{epsfig}
\usepackage[latin1]{inputenc}
\usepackage{amsmath}
\usepackage{amsfonts}
\usepackage{amssymb}
\usepackage{amsthm}
\usepackage{amscd}
\usepackage{verbatim}
\usepackage{caption}
\usepackage{subcaption}
\usepackage{pinlabel}
\usepackage{stmaryrd}
\usepackage{enumerate, enumitem}
\usepackage{todonotes}
\usepackage{bm}
\usepackage{thmtools}
\usepackage{thm-restate}

\usepackage{tikz}
\usetikzlibrary{arrows}
\usetikzlibrary{decorations.pathreplacing}
\usepackage{verbatim}
\usetikzlibrary{cd}

\newtheorem{theorem}{Theorem}[section]

\newtheorem{lemma}[theorem]{Lemma}
\newtheorem{proposition}[theorem]{Proposition}

\newtheorem{corollary}[theorem]{Corollary}

\theoremstyle{definition}
\newtheorem{definition}[theorem]{Definition}

\newtheorem{question}[theorem]{Question}
\theoremstyle{remark}
\newtheorem{remark}[theorem]{Remark}
\newtheorem{example}[theorem]{Example}

\def\F{\mathbb{F}}
\def\N{\mathbb{N}}
\def\Q{\mathbb{Q}}

\def\Z{\mathbb{Z}}

\def\bbT{\mathbb{T}}

\def\cC{\mathcal{C}}

\def\cH{\mathcal{H}}

\def\cR{\mathcal R}
\def\cB{\mathcal{B}}

\def\Sym{\mathrm{Sym}}

\def\CFKi{\CFK^{\infty}}

\def\Im{\operatorname{im}}

\def \gr {\operatorname{gr}}

\def\d{\partial}
\def\varep{\varepsilon}
\def\co{\colon}

\def\spinc{\textrm{Spin}^c}

\def\im{\operatorname{im}}

\def\HF {\mathit{HF}}
\def\HFK {\mathit{HFK}}
\newcommand\HFhat{\widehat{\HF}}

\def\CFK{\mathit{CFK}}

\def\spinc {{\operatorname{spin^c}}}

\def\s{\mathfrak s}

\def\Hom{\mathrm{Hom}}

\def\cCFK{\mathcal{CFK}}

\newcommand{\shift}{\mathrm{sh}} 

\def\Zbang{\Z^!}
\def\lebang{<^!}
\def\gebang{>^!}
\def\leqbang{\leq^!}
\def\geqbang{\geq^!}
\def\bang{\!{}^!}
\def\KL{\mathfrak{K}} 
\def\sgn{\operatorname{sgn}}
\def\Cont{\operatorname{Cont}}
\def\cCFK{\mathcal{CFK}}
\def\CTS{\cC_{\mathit{TS}}}

\def\bfalpha{\boldsymbol{\alpha}}
\def\bfbeta{\boldsymbol{\beta}}
\def\CFKUV{\CFK_\cR} 
\def\bfx{\mathbf{x}}
\def\bfy{\mathbf{y}}
\def\Tors{\operatorname{Tors}}
\def\Span{\operatorname{Span}}

\author[I. Dai]{Irving Dai}
\thanks{The first author was partially supported by NSF grant DGE-1148900.}
\address {Department of Mathematics, Princeton University, Princeton, NJ 08540}
\email{idai@math.princeton.edu}

\author[J. Hom]{Jennifer Hom}
\thanks{The second author was partially supported by NSF grant DMS-1552285 and a Sloan Research Fellowship.}
\address {School of Mathematics, Georgia Institute of Technology, Atlanta, GA 30332}
\email{hom@math.gatech.edu}

\author[M. Stoffregen]{Matthew Stoffregen}
\thanks{The third author was partially supported by NSF grant DMS-1702532.}
\address {Department of Mathematics, Massachusetts Institute of Technology, Cambridge, MA 02142}
\email{mstoff@mit.edu}

\author[L. Truong]{Linh Truong}
\thanks{The fourth author was partially supported by NSF grant DMS-1606451.}
\address {Department of Mathematics, Columbia University, New York, NY 10027}
\email{ltruong@math.columbia.edu}

\numberwithin{equation}{section}

\title[More concordance homomorphisms]{More concordance homomorphisms from knot Floer homology}

\begin{document}

\begin{abstract}
We define an infinite family of linearly independent, integer-valued smooth concordance homomorphisms. Our homomorphisms are explicitly computable and rely on local equivalence classes of knot Floer complexes over the ring $\F[U, V]/(UV=0)$. We compare our invariants to other concordance homomorphisms coming from knot Floer homology, and discuss applications to topologically slice knots, concordance genus, and concordance unknotting number.
\end{abstract}

\maketitle

\section{Introduction}\label{sec:intro}
Beginning with the $\tau$-invariant \cite{OS4ball}, the knot Floer homology package of \linebreak Ozsv\'ath-Szab\'o \cite{OSknots} and independently J. Rasmussen \cite{RasmussenThesis} has had numerous applications to the study of smooth knot concordance. See \cite{Homsurvey} for a survey of such applications.

The goal of this paper is to add to the (already infinite) list of explicitly computable homomorphisms from the smooth knot concordance group $\cC$ to $\Z$:

\begin{theorem}\label{thm:main}
For each $j \in \N$, there is a surjective homomorphism
\[ \varphi_j \co \cC \to \Z. \]
Moreover, 
\[ \bigoplus_{j=1}^\infty  \varphi_j \co \cC \to \bigoplus_{j=1}^\infty  \Z \]
is surjective. In particular, the $\varphi_j$ are linearly independent.
\end{theorem}
\noindent
Our homomorphisms are similar in spirit to Ozsv\'ath-Stipsicz-Szab\'o's $\Upsilon$-invariant, which gives a homomorphism
\[ \Upsilon_K \co \cC \to \Cont([0,2]), \]
where $\Cont([0,2])$ denotes the vector space of piecewise-linear functions on $[0,2]$. Indeed, $\Upsilon$ is defined using $t$-modified knot Floer homology and can be thought of as a generalization of $\tau$ to the $t$-modified knot Floer homology setting. A slight repackaging (by considering the slopes of $\Upsilon_K(t)$) yields a $\Z$-valued homomorphism for each rational value of $t$. Similarly, our invariants can be thought of as a generalization of $\tau$ to a shifted version of knot Floer homology. The homomorphisms $\varphi_j$ are then certain linear combinations of $\tau$ associated to shifted knot Floer homology. Just as $\tau$ can be recovered from $\Upsilon(t)$, it can also be recovered from $\varphi_j$:

\begin{restatable}{proposition}{proptau}\label{prop:tau}
Let $K$ be a knot in $S^3$. Then we have the following equality relating the Ozsv\'ath-Szab\'o $\tau$-invariant with $\varphi_j$:
\[ \tau(K) = \sum_{j \in \N} j\varphi_j(K). \]
\end{restatable}

Both $\Upsilon(t)$ and $\varphi_j$ factor through the local equivalence group of knot Floer complexes (\cite[Theorem 1.5]{Zemkeconnsuminv}, forgetting the involutive part; equivalently stable equivalence from \cite[Theorem 1]{Homsurvey}; equivalently $\nu^+$-equivalence of \cite{KimPark}). Following \cite[Section 3]{Zemkeconnsuminv}, the knot Floer complex can be viewed as a module over $\F[U, V]$; local equivalence is then an equivalence relation between certain such complexes. In our setting, the invariants $\varphi_j$ actually factor through the local equivalence group defined over the ring $\F[U, V]/(UV=0)$, which is the same as the group constructed using $\varep$-equivalence in \cite[Definition 1]{Homconcordance}. The advantage of quotienting by $UV=0$ is that the resulting local equivalence group is totally ordered; this total order is the same as the order induced by $\varep$ \cite{Hominfiniterank}. Using this order, we have the following characterization result.

\begin{theorem}\label{thm:localequivlex}
Every knot Floer complex coming from a knot in $S^3$ is locally equivalent mod $UV$ to a standard complex (defined in Section \ref{subsec:standards}) and can be completely described by a finite (symmetric) sequence of nonzero integers $(a_i)_{i=1}^{2n}$. Moreover, if we endow the integers with the following unusual order
\[ -1 \lebang -2 \lebang -3 \lebang \dots \lebang  0 \lebang \dots \lebang 3 \lebang 2 \lebang 1, \]
then local equivalence classes mod $UV$ are ordered lexicographically with respect to their standard representatives.
\end{theorem}

\subsection{Properties of $\varphi_j$}
The homomorphisms $\varphi_j$ have many properties in common with $\Upsilon$: both invariants take a particularly simple form on homologically thin knots and L-space knots. We use the convention that $K$ is an \emph{L-space knot} if $K$ admits a positive L-space surgery. 

\begin{restatable}{proposition}{propthin}\label{prop:thin}
If $K$ is homologically thin, then 
\[ \varphi_j(K) = \begin{cases}
	\tau(K) &\text{ if } j=1 \\
	0 &\text{ otherwise.}
\end{cases}
\]
\end{restatable}

\begin{restatable}{proposition}{propLspace}\label{prop:Lspace}
Let $K$ be an L-space knot with Alexander polynomial 
\[ \Delta_K(t) = \sum_{i=0}^n (-1)^i t^{b_i}\]
where $(b_i)_{i=0}^n$ is a decreasing sequence of integers and $n$ is even. Define
\[ c_i = b_{2i-2} - b_{2i-1}, \quad 1 \leq i \leq n/2. \]
Then
\[ \varphi_j(C) = \# \{ c_i \mid c_i =j \}. \]
\end{restatable}

\begin{example}
Consider the torus knot $T_{3,4}$. We have that $\Delta_{T_{3,4}}(t) = t^6-t^5+t^3-t+1$, and so by Proposition \ref{prop:Lspace}, we have 
\[ \varphi_j(T_{3,4}) = \begin{cases}
	1 &\text{ if } j=1,2 \\
	0 &\text{ otherwise.}
\end{cases}
\]
See Figure \ref{fig:T34} for a visual depiction of $\CFKi(T_{3,4})$. 
\end{example}

\begin{example} \label{ex:torusknots}
More generally, the torus knot $T_{n, n+1}$ has Alexander polynomial 
\[ \Delta_{T_{n, n+1}}(t) = \sum_{i=0}^{n-1}t^{ni}  - \sum_{i=0}^{n-2} t^{ni+i+1} \]
which yields $(c_i)_{i=1}^{n-1} = (1, 2, \dots, n-1)$. Thus
\[ \varphi_j(T_{n, n+1}) = \begin{cases}
	1 &\text{ if } j=1,2, \dots, n-1 \\
	0 &\text{ otherwise.}
\end{cases}
\]
\end{example}

\begin{remark}
Note that if $K$ is an L-space knot, then by Proposition \ref{prop:Lspace}, $\varphi_j(K) \geq 0$ for all $j$. This provides an easy (although fairly weak) method for showing that a linear combination of knots is not concordant to any L-space knot.
\end{remark}

\begin{remark}
In Propositions \ref{prop:thin} and \ref{prop:Lspace} (as well as in the above examples), note that $\varphi_j$ is the (signed) count of the number of horizontal arrows of length $j$. We will see in Definition \ref{def:homs} that $\varphi_j$ is equal to the signed count of horizontal arrows in the standard complex representative of $K$ (in the sense of Theorem~\ref{thm:localequivlex}).
\end{remark}

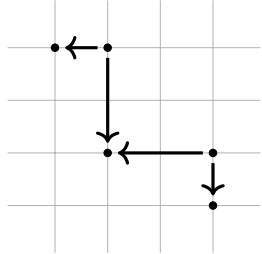
\begin{figure}[ht]
\centering
\begin{tikzpicture}[scale=0.7]
	\draw[step=1, black!30!white, very thin] (-0.9, -0.9) grid (3.9, 3.9);
	\filldraw (0, 3) circle (2pt) node[] (a){};
	\filldraw (1, 3) circle (2pt) node[] (b){};
	\filldraw (1, 1) circle (2pt) node[] (c){};
	\filldraw (3, 1) circle (2pt) node[] (d){};
	\filldraw (3, 0) circle (2pt) node[] (e){};
	\draw [very thick, <-] (a) -- (b);
	\draw [very thick, <-] (c) -- (b);
	\draw [very thick, <-] (c) -- (d);
	\draw [very thick, <-] (e) -- (d);
\end{tikzpicture}
\caption{The knot Floer complex of $T_{3,4}$.}
\label{fig:T34}
\end{figure}

While $\Upsilon(t)$ and $\varphi_j$ have many properties in common, there do exist knots $K$ for which $\Upsilon_K(t) \equiv 0$ while $\varphi_j(K)$ is nontrivial. Let $K_{p,q}$ denote the $(p,q)$-cable of $K$, where $p$ denotes the longitudinal winding.

\begin{proposition}
Let $K = T_{2,5} \# -T_{4,5} \# T_{2,3;2,5}$. Then $\Upsilon_K(t) \equiv 0$, while
\[ \varphi_j(K) = \begin{cases}
	2 &\text{ if } j=1 \\
	-1 &\text{ if } j=2 \\
	0 &\text{ otherwise.}
\end{cases}
\]
\end{proposition}

\begin{proof}
The fact that $\Upsilon_K(t) \equiv 0$ follows from the proof of \cite[Theorem 2]{HomUpsilon}. The computation of $\varphi_j(K)$ follows from Proposition \ref{prop:Lspace} and the fact that the $\varphi_j$ are homomorphisms. (Note that $T_{2,3;2,5}$ is an L-space knot; see the proof of \cite[Lemma 2.1]{HomUpsilon} for the relevant Alexander polynomial.)
\end{proof}

Conversely, while we do not have an explicit topological example, there is no algebraic obstruction to the existence of knots with $\varphi_j(K)$ trivial and $\Upsilon_K(t)$ nontrivial.

\begin{proposition}
Suppose there existed a knot $K$ whose knot Floer complex was given by Figure \ref{fig:OSScomplex}. Then $\Upsilon_K(t)$ is nontrivial, while $\varphi_j(K) = 0$ for all $j$.
\end{proposition}

\begin{figure}[ht]
\centering
\begin{tikzpicture}[scale=0.7]
	\draw[step=1, black!30!white, very thin] (-0.9, -0.9) grid (3.9, 3.9);
	\filldraw (0, 3) circle (2pt) node[] (a){};
	\filldraw (3, 3) circle (2pt) node[] (b){};
	\filldraw (3, 0) circle (2pt) node[] (c){};
	\filldraw (0, 0) circle (2pt) node[] (d){};
	\filldraw (1, 1) circle (2pt) node[] (e){};
	\draw [very thick, <-] (a) -- (b);
	\draw [very thick, <-] (c) -- (b);
	\draw [very thick, ->] (c) -- (d);
	\draw [very thick, ->] (a) -- (d);
	\draw [very thick, ->] (e) -- (d);
\end{tikzpicture}
\caption{The complex from \cite[Figure 6]{OSS}.}
\label{fig:OSScomplex}
\end{figure}
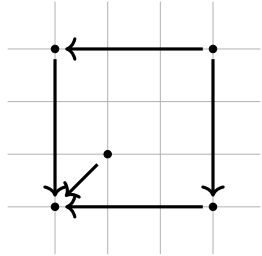

\begin{proof}
The computation of $\Upsilon_K(t)$ is given in \cite[Proposition 9.4]{OSS}. Since diagonal arrows vanish modulo $UV$, it is easily checked that the above complex is trivial in local equivalence (see Section~\ref{sec:knotlike}). This implies that $\varphi_j(K) = 0$ for all $j$.
\end{proof}

\subsection{Topological applications of $\varphi_j$}
The homomorphisms $\varphi_j$ have applications to $\CTS$, the subgroup of $\cC$ generated by topologically slice knots. (That is, $\CTS$ is generated by knots bounding locally flat disks in $B^4$.) Let $D$ denote the positively-clasped, untwisted Whitehead double of $T_{2,3}$, and let $K_n = D_{n, n+1} \# -T_{n, n+1}$.

\begin{theorem}\label{thm:TS}
Consider the topologically slice knots $K_n$ described above. For each index $n$, we have $\varphi_n(K_n) = 1$ and $\varphi_j(K_n) = 0$ for all $j > n$. In particular, the homomorphisms 
\[ \bigoplus_{j=1}^\infty  \varphi_j \co \CTS \to \bigoplus_{j=1}^\infty  \Z \]
map the span of the $K_n$ isomorphically onto $\bigoplus_{j=1}^\infty  \Z$.
\end{theorem}

\begin{remark}
The knots $K_n$ are the same knots considered in \cite{Hominfiniterank}. However, there is an error in the proof of the main result of \cite{Hominfiniterank}. Fortunately, the above theorem shows that the knots $K_n$ do in fact generate an infinite-rank summand of $\CTS$. Moreover, they show this in a way that preserves the spirit of \cite{Hominfiniterank}, namely by considering knot Floer complexes modulo $\varep$-equivalence and extracting numerical invariants based on the lengths of vertical and horizontal arrows.
\end{remark}

We also have applications of $\varphi_j$ to concordance genus and concordance unknotting number. Recall that the \emph{concordance genus} of $K$ is defined to be
\[ g_c(K) = \min \{ g(K') \mid K, K' \text{ smoothly concordant} \} \]
where $g(K')$ denotes the Seifert genus of $K'$. Note that 
\[ g_c(K) \geq g_4(K) \]
where $g_4(K)$ denotes the smooth four-ball genus of $K$. The \emph{concordance unknotting number} of $K$ is defined to be
\[ u_c(K) = \min \{ u(K') \mid K, K' \text{ smoothly concordant} \} \]
where $u(K')$ denotes the unknotting number of $K'$.  Note that again,
\[ u_c(K) \geq g_4(K). \]
Since $g_4(K) \geq |\tau(K)|$, the knot Floer homology of $K$ provides lower bounds on both $g_c(K)$ and $u_c(K)$. Here, we show that the invariants $\varphi_j$ bound concordance genus and concordance unknotting number as follows:

\begin{theorem}\label{thm:gcuc}
Let
\[N(K) = \begin{cases}
0 & \text{if } \varphi_j(K)=0 \text{ for all } j
\\
\max \{ j \mid \varphi_j(K) \neq 0\} & otherwise.
\end{cases}\]
Then 
\begin{enumerate}
	\item \label{it:gc} $g_c(K) \geq \frac{1}{2}N(K)$, and
	\item \label{it:uc} $u_c(K) \geq N(K)$.
\end{enumerate}
\end{theorem}
\noindent
Let $\Tors_U M$ denotes the $U$-torsion submodule of an $\F[U]$-module $M$. The quantity $N(K)$ is bounded above by the maximal order of an element in $\Tors_U \HFK^-(K)$, as follows:

\begin{restatable}{proposition}{propHFKm}\label{prop:HFKm}
If $U^M \cdot \Tors_U \HFK^-(K) = 0$, then $\varphi_j(K) = 0$ for all $j > M$. In particular, $N(K) \leq M$.
\end{restatable}

The bounds in Theorem \ref{thm:gcuc} \eqref{it:uc} are sharp (e.g., for the trefoil); it is unknown to the authors whether the bound in Theorem \ref{thm:gcuc} \eqref{it:gc} is sharp. Note that in many cases, the bounds are rather weak; for example, $N(T_{n, n+1}) = n-1$, while $g_4(T_{n, n+1}) = \tau(T_{n, n+1}) = n(n-1)/2$. The proof of the concordance genus bound in Theorem \ref{thm:gcuc} \eqref{it:gc} is similar to the proof of \cite[Theorem 2]{Homconcordancegenus}, and indeed is strong enough to recover \cite[Theorem 3]{Homconcordancegenus}. The proof of Theorem \ref{thm:gcuc} \eqref{it:uc} relies on unknotting number bounds from \cite{AlishahiEftekharyunknotting}.

We have the following application of Theorem \ref{thm:gcuc} \eqref{it:uc}.

\begin{theorem}\label{thm:topsliceuc}
There exist topologically slice knots $\{K_n\}_{n = 1}^{\infty}$ such that $g_4(K_n) = 1$ for all $n$, while $u_c(K_n) \geq n.$
\end{theorem}

\noindent The knots used to prove Theorem \ref{thm:topsliceuc} are the same knots appearing in \cite[Theorem 3]{Homconcordancegenus}. In \cite{OwensStrle}, Owens-Strle give examples of knots for which $u_c(K) - g_4(K) = 1$. As far as the authors know, Theorem \ref{thm:topsliceuc} gives the first known examples of knots for which $u_c(K) - g_4(K)$ is arbitrarily large.

\subsection{Remarks}
We conclude with a few remarks relating the present work with other results. In \cite{OSborderedalgebras}, Ozsv\'ath-Szab\'o define a bordered-algebraic knot invariant which is isomorphic to the knot Floer complex over the ring $\F[U, V]/(UV=0)$. Their bordered-algebraic knot invariant is particularly amenable to computer computation. It should thus be possible to implement an effective computer program to calculate the homomorphisms $\varphi_j$.

Theorem \ref{thm:char} is closely related to horizontally and vertically simplified bases for the knot Floer complex, defined in \cite[Section 11.5]{LOT}. Indeed, Corollary \ref{cor:splitting} states every knot Floer complex over $\F[U, V]/(UV=0)$ contains a direct summand with a simultaneously vertically and horizontally simplified basis, and that this summand supports $\HF^\infty(S^3)$. This is closely related to the notion of loop-type modules, defined in \cite[Definition 3.1]{HanselmanWatsonCalculus}. (Note that over the ring $\F[U, V]$, not every complex admits a  simultaneously vertically and horizontally simplified basis; see \cite[Figure 3]{Hominfiniterank}.)

Lastly, we point out that the techniques in this paper are the knot Floer analogues of the techniques used in \cite{DHSThomcobord} to study the three-dimensional homology cobordism group.

\subsection*{Organization}
In Section \ref{background}, we briefly recall the definition of the knot Floer complex, working over the ring $\cR = \F[U, V]/(UV = 0)$. In Section \ref{sec:knotlike}, we introduce the notion of a knot-like complex, and define the local equivalence group $\KL$ of knot-like complexes. In Section \ref{sec:standards}, we define a particularly simple family of knot-like complexes, which we call standard complexes. We use these to construct a sequence of numerical invariants associated to any knot-like complex in Section \ref{sec:ai}. This is used in Section \ref{sec:char} to show that every knot-like complex is locally equivalent to a standard complex. In Section \ref{sec:homs}, we apply our characterization of knot-like complexes to define the homomorphisms $\varphi_j$. In Section \ref{sec:thinLspace}, we prove Propositions \ref{prop:thin} and \ref{prop:Lspace} (computing $\varphi_j$ for thin and L-space knots). In Section \ref{sec:topslice}, we prove Theorem \ref{thm:TS} (on an infinite-rank summand of $\CTS$), and in Section \ref{sec:gcuc}, we prove Theorems \ref{thm:gcuc} and \ref{thm:topsliceuc} (on applications of $\varphi_j$ to $g_c$ and $u_c$). Finally, we conclude with some further remarks and open questions in Section \ref{sec:furtherremarks}.

Throughout, we work over $\F = \Z/2\Z$. We use the convention that $\N = \Z_{>0}$.

\subsection*{Acknowledgements}
We would like to thank Akram Alishahi, Tye Lidman, Chuck Livingston, Brendan Owens, and Ian Zemke for helpful conversations.

\section{Background on knot Floer homology}\label{background}
In this section, we give a brief overview of knot Floer homology, primarily to establish notation. We assume that the reader is familiar with knot Floer homology as in \cite{OSknots} and \cite{RasmussenThesis}; see \cite{ManolescuKnotIntro} and \cite{Homsurvey} for survey articles on this subject. Our conventions mostly follow those in \cite{Zemkeabsgr}; see, in particular, Section 1.5 of \cite{Zemkeabsgr}.

\begin{definition}
Let $\cR = \F[U, V]/(UV=0)$, endowed with a relative bigrading $\gr = (\gr_U, \gr_V)$, where $\gr(U) = (-2, 0)$ and $\gr(V) = (0, -2)$. We call $\gr_U$ the \emph{$U$-grading} and $\gr_V$ the \emph{$V$-grading}.
\end{definition}

Let $\cH = (\Sigma, \bfalpha, \bfbeta, w, z)$ be a doubly-pointed Heegaard diagram compatible with $(S^3, K)$. Define $\CFKUV(\cH)$ to be the chain complex freely generated over $\cR$ by $\bfx \in \bbT_{\bfalpha} \cap \bbT_{\bfbeta}$ with differential
\[ \d \bfx = \sum_{\bfy \in \bbT_{\bfalpha} \cap \bbT_{\bfbeta}}  \sum_{\substack{\phi \in \pi_2(\bfx, \bfy) \\ \mu(\phi)=1}} U^{n_w(\phi)} V^{n_z(\phi)} \bfy, \]
where, as usual, $\pi_2(\bfx, \bfy)$ denotes homotopy classes of disks in $\Sym^g (\Sigma)$ connecting $\bfx$ to $\bfy$, and $\mu(\phi)$ denotes the Maslov index of $\phi$. The chain complex $\CFKUV(\cH)$ comes equipped with a relative bigrading $\gr = (\gr_U, \gr_V)$, defined as follows. Given $\bfx, \bfy \in \bbT_{\bfalpha} \cap \bbT_{\bfbeta}$ and $\phi \in \pi_2(\bfx, \bfy)$, let the relative grading shifts be given by
\begin{align*}
	\gr_U(\bfx, \bfy) &= \mu(\phi) -2n_w(\phi) \\
	\gr_V(\bfx, \bfy) &= \mu(\phi) -2n_z(\phi). 
\end{align*}
It follows that the differential has degree $(-1, -1)$. (In the literature, $\gr_U$ is usually referred to as \emph{Maslov grading}.) We define a relative \emph{Alexander grading} by 
\[ A(\bfx, \bfy) = \frac{1}{2} (\gr_U(\bfx, \bfy) - \gr_V(\bfx, \bfy)) = n_z(\phi) - n_w(\phi). \]
Note that the variable $U$ lowers $\gr_U$ by $2$, preserves $\gr_V$, and lowers $A$ by $1$. The variable $V$ preserves $\gr_U$, lowers $\gr_V$ by $2$, and increases $A$ by $1$. The differential preserves the Alexander grading.

Up to chain homotopy over $\cR$, the chain complex $\CFKUV(\cH)$ is an invariant of $K \subset S^3$, and so we will typically write $\CFKUV(K)$ rather than $\CFKUV(\cH)$. We now recall some facts from \cite{OSknots}. The complex $\CFKUV(K)$ has the following symmetry property. Let $\overline{\CFKUV}(K)$ denote the complex obtained by interchanging the roles of $U$ and $V$. (Note that we thus also interchange the values of $\gr_U$ and $\gr_V$.) Then
\[ \CFKUV(K) \simeq \overline{\CFKUV}(K). \]
The knot Floer complex behaves nicely with respect to connected sums. Indeed, we have that
\[ \CFKUV(K_1 \# K_2) \simeq \CFKUV(K_1) \otimes_\cR \CFKUV(K_2). \]
We also have that
\[ \CFKUV(-K) \simeq \CFKUV(K)^\vee, \]
where $\CFKUV(K)^\vee = \Hom_\cR(\CFKUV(K), \cR)$.

\begin{remark}
Since the differential preserves the Alexander grading, the complex $\CFKUV$ splits  -- as a chain complex over $\F$, but \emph{not} as an $\cR$-module -- as a direct sum over the Alexander grading:
\[ \CFKUV (K) = \bigoplus_{s \in \Z} \CFKUV (K, s), \]
where 
\begin{align*}
U &\co  \CFKUV (K, s) \to  \CFKUV (K, s-1) \\
V &\co  \CFKUV (K, s) \to  \CFKUV (K, s+1).
\end{align*}
The chain complex $\CFKUV (K, s)$ is isomorphic to the complex $\widehat{A}_s$ from \cite{OSinteger}; that is, $H_*(\CFKUV (K, s))$ is isomorphic (as a relatively graded vector space) to $\smash{\HFhat(S^3_N(K), \s_s)}$, the Heegaard Floer homology of large surgery on $K$ in the $\spinc$ structure corresponding to $s$.
\end{remark}

The version of knot Floer homology we have constructed here follows slightly different conventions than the usual definition in e.g.\ \cite{OSknots}. For the convenience of the reader, we recall some of the most salient features of the standard knot Floer homology package, and explicitly translate them into our setting. For further discussion, see Section 1.5 of \cite{Zemkeabsgr}. 

First, consider the $\F$-vector space $\smash{\widehat\HFK(K)}$, which is defined by not allowing holomorphic disks in the definition of $\d$ to cross either the $w$ or the $z$ basepoint. In our context, this is isomorphic to $H_*(\CFKUV(K)/(U, V))$, where $(U,V)$ denotes the ideal generated by $U$ and $V$. The Alexander grading is given by $\smash{A = \frac{1}{2} (\gr_U-\gr_V)}$ and the Maslov grading is given by $M = \gr_U$. 

Next, consider the $\F[U]$-module $\HFK^-(K)$, which is defined by taking the homology of the associated graded complex of $\CFK^-(K)$ with respect to the Alexander filtration. This is equivalent to allowing holomorphic disks to cross the $w$ but not the $z$ basepoint. In our context, this yields $H_*(\CFKUV(K)/V)$, where again the Alexander grading is given by $\smash{A = \frac{1}{2} (\gr_U-\gr_V)}$ and the Maslov grading is given by $M = \gr_U$. It is a standard fact that for knots in $S^3$, the $\F[U]$-module $\HFK^-(K) \cong H_*(\CFKUV(K)/V)$ has a single $U$-nontorsion tower.\footnote{By this, we mean that $H_*(\CFKUV(K)/V)/U\text{-torsion} \cong \F[U]$. Note, however, that this copy of $\F[U]$ is not required to be generated by an element with $\gr_U = 0$.} By symmetry, it follows that $H_*(\CFKUV(K)/U)$ has a single $V$-nontorsion tower. 

We now claim that these two nontorsion towers satisfy the following grading normalizations:
\begin{enumerate}
\item The $U$-gradings of all $V$-nontorsion classes in $H_*(\CFKUV(K)/U)$ are zero.\label{lab:cond1}
\item The $V$-gradings of all $U$-nontorsion classes in $H_*(\CFKUV(K)/V)$ are zero.\label{lab:cond2}
\end{enumerate}
Note that all $V$-nontorsion classes in $H_*(\CFKUV(K)/U)$ have the same $U$-grading, since multiplication by $V$ does not change $\gr_U$. Similarly, all $U$-nontorsion classes in $H_*(\CFKUV(K)/V)$ have the same $V$-grading. To see the claim, consider the complex $\CFKUV(K)$ and set $U = 0$ and $V = 1$. This means that we allow holomorphic disks to cross the $z$ but not the $w$ basepoint, and we disregard the Alexander filtration. This yields a complex whose homology computes $\smash{\widehat\HF(S^3)} \cong \F$, which is concentrated in Maslov grading zero. Using the fact that the Maslov grading is equal to $\gr_U$, some thought shows that the $V$-nontorsion tower of $H_*(\CFKUV(K)/U)$ is thus generated by an element with $\gr_U = 0$. By symmetry, we likewise have that any $U$-nontorsion element in $H_*(\CFKUV(K)/V)$ has $\gr_V = 0$.

Finally, recall that the concordance invariant $\tau(K)$ is defined to be the negative of the maximal Alexander grading of any $U$-nontorsion element in $\HFK^-(K) \cong H_*(\CFKUV(K)/V)$. By the previous two paragraphs, this means that
\[
\tau(K) = - \max\{\dfrac{1}{2}\gr_U(x) \mid x \in H_*(\CFKUV(K)/V) \text{ is not } U \text{-torsion} \}.
\]
By symmetry, we conclude that similarly
\[
\tau(K) = - \max\{ \dfrac{1}{2}\gr_V(x) \mid x \in H_*(\CFKUV(K)/U) \text{ is not } V \text{-torsion} \}.
\]
The reader should think of the complexes $\CFKUV(K)/U$ and $\CFKUV(K)/V$ as deleting horizontal and vertical arrows (respectively) in the pictorial representation of $\CFKUV$. It may be helpful to keep in mind Figure~\ref{fig:T34}. There, the $V$-nontorsion tower of $H_*(\CFKUV(K)/U)$ is generated by the top-left basis element, while the the $U$-nontorsion tower of $H_*(\CFKUV(K)/V)$ is generated by the bottom-right basis element. 

The following definition is particularly useful in applications of knot Floer homology to concordance:

\begin{definition}\label{def:loceq}
Let $K_1$ and $K_2$ be knots in $S^3$. We say that $\CFKUV(K_1)$ and $\CFKUV(K_2)$ are \emph{locally equivalent} if there exist absolutely $U$-graded, absolutely $V$-graded $\cR$-equivariant chain maps
\[ f \co \CFKUV(K_1) \to \CFKUV(K_2) \quad \text{ and }  \quad g \co \CFKUV(K_2) \to \CFKUV(K_1) \]
such that $f$ and $g$ induce isomorphisms on $H_*(\CFKUV(K_i)/U)/V\text{-torsion}$. Roughly speaking, this means that $f$ maps the top of the $V$-tower in $H_*(\CFKUV(K_1)/U)$ to the top of the $V$-tower in $H_*(\CFKUV(K_2)/U)$, and vice-versa for $g$.
\end{definition}

\noindent Local equivalence is considered in the involutive setting in \cite[Section 2.3]{Zemkeconnsuminv}.

\begin{remark}
Note that $\CFKUV(K)$ is locally equivalent to $\CFKUV(O)$, where $O$ denotes the unknot, if and only if  $\CFKUV(K) \simeq \CFKUV(O) \oplus A$, where $A$ is a chain complex over $\cR$ with $U^{-1}H_*(A) = V^{-1}H_*(A) = 0$. It is straightforward to verify that local equivalence over $\cR$ and $\varep$-equivalence (see \cite[Section 2]{Hominfiniterank}) are the same (after translating between $\cR$-modules and bifiltered chain complexes over $\F[U, U^{-1}]$).
\end{remark}

\begin{theorem}[{\cite[Theorem 1.5]{Zemkeconnsuminv}, cf.\ \cite[Theorem 2]{Homconcordance}}] \label{thm:concle}
If $K_1$ and $K_2$ are concordant, then $\CFKUV(K_1)$ and $\CFKUV(K_2)$ are locally equivalent.
\end{theorem}
\noindent
Theorem~\ref{thm:concle} follows from \cite[Theorem 1.5]{Zemkeconnsuminv} by forgetting the involutive component and quotienting by $UV$, or from \cite[Theorem 2]{Homconcordance} by translating from $\varep$-equivalence and bifiltered chain complexes to local equivalence and $\cR$-modules.
\section{Knot-like complexes and their properties}\label{sec:knotlike}
In this section, we consider abstract $\cR$-complexes satisfying many of the same formal properties as $\CFKUV(K)$. We show that modulo local equivalence, the set of such complexes forms a group, with the operation induced by tensor product. Moreover, we show that this group is totally ordered.

\subsection{Knot-like complexes}
We begin by defining knot-like complexes, so named because they are $\cR$-complexes satisfying many of the properties of $\CFKUV$ from the previous section.

\begin{definition}\label{def:knotlike}
A \emph{knot-like complex} $C$ is a free, finitely generated, bigraded chain complex over $\cR$ such that
\begin{enumerate}
	\item $H_*(C/U)$ has a single $V$-nontorsion tower, lying in $\gr_U = 0$.
	\item $H_*(C/V)$ has a single $U$-nontorsion tower, lying in $\gr_V = 0$.
\end{enumerate}
Again, we mean by this that $H_*(C/U)/V\text{-torsion}$ is isomorphic to $\F[V]$, and that all of the $V$-nontorsion elements in $H_*(C/U)$ have $U$-grading zero. A similar statement holds for $H_*(C/V)$. The differential $\d$ is required to have degree $(-1, -1)$.
\end{definition}

\begin{remark}
Note that we do \textit{not} in general require any symmetry with respect to interchanging $U$ and $V$.
\end{remark}

\begin{definition}\label{def:le}
Let $C_1$ and $C_2$ be two knot-like complexes. We say that $C_1 \leq C_2$ if there exists an absolutely $U$-graded, relatively $V$-graded $\cR$-equivariant chain map
\[ f \co C_1 \rightarrow C_2 \]
such that $f$ induces an isomorphism on $H_*(C_i/U)/V\text{-torsion}$. We call $f$ a \emph{local map}. We say that two knot-like complexes $C_1$ and $C_2$ are \emph{locally equivalent}, denoted $C_1 \sim C_2$, if $C_1 \leq C_2$ and $C_2 \leq C_1$. 
\end{definition}
\noindent
We will also occasionally use the terminology:
\begin{definition}
Let $C$ be a knot-like complex and let $x \in C$. We say that $x$ is a \textit{$V$-tower class} if $[x]$ is a maximally $V$-graded, $V$-nontorsion cycle in $H_*(C/U)$. Similarly, we say that $x$ is a \textit{$U$-tower class} if $[x]$ is a maximally $U$-graded, $U$-nontorsion cycle in $H_*(C/V)$. Thus $f$ (as defined above) sends $V$-tower classes to $V$-tower classes.
\end{definition}

\begin{remark}\label{remark:absrel}
Note that $f$ in Definition~\ref{def:le} is \textit{not} required to be absolutely $V$-graded, but rather only relatively $V$-graded. Thus, \textit{a priori} the notion of local equivalence in Definition~\ref{def:le} is strictly weaker than the notion of local equivalence presented in Definition~\ref{def:loceq}; i.e., we might have two knot-like complexes $C_1$ and $C_2$ which are locally equivalent via maps $f$ and $g$ that introduce complementary $V$-grading shifts. However, we will show in Lemma~\ref{lem:loceqsymmetric} that if $C_1$ and $C_2$ are locally equivalent (in the sense of Definition~\ref{def:le}) via $f$ and $g$, then $f$ and $g$ induce isomorphisms on $H_*(C_i/V)/U\text{-torsion}$ (i.e., send $U$-tower classes to $U$-tower classes), even without any symmetry requirements on the $C_i$. Combined with the normalization conventions of Definition~\ref{def:knotlike}, this shows that $f$ and $g$ are absolutely $V$-graded.
\end{remark}

It is straightforward to verify that $\leq$ is a partial order on the set of local equivalence classes of knot-like complexes.

\begin{remark}
Our notion of local equivalence agrees with \cite[Definition 2.4]{Zemkeconnsuminv} after forgetting $\iota_K$ and modding out by the ideal generated by $UV$. This definition of local equivalence also agrees with the equivalence relation defined using $\varepsilon$ from \cite[Section 4.1]{Homconcordance}; for this, see Theorem~\ref{thm:char} and Corollary~\ref{cor:splitting}.
\end{remark}

Let $(U, V)$ denote the ideal generated by $U$ and $V$. If $C$ is a free, finitely generated chain complex over $\cR$, then every element $x$ in $(U, V)$ can be uniquely expressed as $x_U + x_V$, where $x_U \in \Im U$ and $x_V \in \Im V$. 

\begin{definition}
We say a chain complex over $\cR$ is \emph{reduced} if $\d \equiv 0 \mod (U, V)$. In a reduced complex, we can write $\d$ as the sum $\d = \d_U + \d_V$, where if $\d x = y$, then $\d_U x = y_U$ and $\d_V x = y_V$. Note that $\d_U^2 = \d_V^2 = 0$. We call $\d_U$ the \textit{$U$-differential} and refer to elements with $\d_U x = 0$ as \textit{$U$-cycles}; similarly, we call $\d_V$ the \textit{$V$-differential} and refer to elements with $\d_V x = 0$ as \textit{$V$-cycles}.
\end{definition}

\begin{lemma}\label{lem:reduced}
Every knot-like complex $C$ is locally equivalent to a reduced knot-like complex $C'$.
\end{lemma}

\begin{proof}
Suppose that $C$ is not reduced. Then there exists $x \in C$ such that $\d x$ is not in the ideal generated by $U$ and $V$. We claim that we may complete $\{ x, \d x\}$ to a basis $\{ x, \d x, y_1, \dots, y_n \}$ for $C$ such that the $y_i$ generate a subcomplex $C'$ of $C$. To see this, first complete $\{x, \d x\}$ to an $\cR$-basis $\{ x, \d x, y_1, \dots, y_n \}$ for $C$, where $\d$ does not necessarily preserve the span of the $y_i$. Here, we are using the fact that if $N$ is a (free) submodule of a free module $M$, then a basis for $N$ can be extended to a basis for $M$ if and only if $M/N$ is also free. To apply this in our case, note that $x$ and $\d x$ do not lie in the image of $(U, V)$. A grading argument then shows that no linear combination of $x$ and $\d x$ lies in the image of $(U, V)$.

For each $y_i$, we then write $\partial y_i$ as a linear combination of $x$, $\d x$, and the other basis elements $y_j$. By adding multiples of $x$ to $y_i$, we may assume that $\d x$ does not appear in any differential $\d y_i$. This also shows that $x$ does not appear in $\d y_i$, since then we would have 
\[
0 = \d^2 y_i = \d\left(P(U, V)x + \sum P_j(U, V) y_j \right)
\]
for some polynomials $P(U, V)$ and $P_j(U, V)$, which would imply that $\d x$ appears in some $\d y_j$.

It follows that 
\[ 0 \to \langle x, \d x \rangle \to C \xrightarrow{p} C' \to 0 \]
is a split short exact sequence of freely generated $\cR$-complexes. Since $\langle x, \d x \rangle$ is acyclic by construction, the projection $p \co C \to C'$ and section $s \co C' \to C$ both induce isomorphisms on homology. Hence $C$ and $C'$ are locally equivalent. Since $C$ is finitely generated, we may iterate this procedure to arrive at a reduced complex.
\end{proof}

From now on, we will assume that all of our knot-like complexes are reduced.

\subsection{The local equivalence group of knot-like complexes}
We now show that knot-like complexes modulo local equivalence form a group, with the operation induced by tensor product. Moreover, we will show that the partial order $\leq$ is in fact a total order. We begin with some routine formalism:

\begin{definition}
The \emph{product} of two knot-like complexes $C_1$ and $C_2$ is $C_1 \otimes_\cR C_2$.
\end{definition}

\begin{lemma}
The product of two knot-like complexes is a knot-like complex.
\end{lemma}

\begin{proof}
Straightforward.
\end{proof}

\begin{definition}
Let $\KL$ denote the set of local equivalence classes of knot-like complexes, with the operation induced by $\otimes$.
\end{definition}

\begin{proposition}
The pair $(\KL, \otimes)$ forms an abelian group.
\end{proposition}

\begin{proof}
This is straightforward to verify.  The identity is given by $\cR$ with trivial differential, and the inverse of $[C]$ is $[C^\vee]$, where $C^\vee = \Hom_\cR(C, \cR)$. 
\end{proof}

\begin{remark}
See \cite[Proposition 2.6]{Zemkeconnsuminv} for the analogous result in the involutive setting over the ring $\F[U, V]$.
\end{remark}

We now come to the significantly more interesting proposition. 
\begin{proposition}\label{prop:totalorder}
The relation $\leq$ defines a total order on $\KL$.
\end{proposition}
\noindent
Proposition \ref{prop:totalorder} is a consequence of the following lemma.

\begin{lemma}\label{lem:totalorder}
Let $C$ be a knot-like complex. If there does not exist a local map $f \co \cR \to C$, then there exists a local map $g \co C \to \cR$.
\end{lemma}

\begin{proof}
The idea of the proof is we build a basis $\{x, t_i\}$ for $C$ such that quotienting by the span of $\{t_i\}$ gives the desired local map. Roughly, we first find a basis for the subcomplex $A$ generated by elements $w$ such that some $U$-power of $w$ is in the image of $\d_U$ or some $V$-power of $w$ is in the image of $\d_V$. We then extend this basis by an element $x$ representing a $V$-nontorsion class in $H_*(C/U)$. We use the absence of a local map from $\cR$ to $C$ in order to guarantee that $x$ is not in $A$. Finally, we complete this to a basis for all of $C$. We describe this argument more precisely below.

We begin by finding a ``vertically simplifed" basis for $C$ which is especially nice with respect to $\d_V$. Since $\F[V] \cong \cR/U$ is a PID, the complex $C/U$ admits a basis $\cB = \{x, y_i, z_i\}$ over $\F[V]$ such that
\begin{align*}
&\d_V x = 0, \\
&\d_V y_i = V^{\eta_i} z_i, \text{ and} \\
&\d_V z_i = 0
\end{align*}
for some set of positive integers $\eta_i$. Since $C$ is a free $\cR$-module, it is easily checked that choosing any lift of $\cB$ from $C/U$ to $C$ yields an $\cR$-basis for $C$, which (by abuse of notation) we also denote by $\cB = \{x, y_i, z_i\}$. Moreover, since $UV = 0$, these elements also satisfy the equalities $\d_V x = 0$, $\d_V y_i = V^{\eta_i} z_i$, and $\d_V z_i = 0$. We will henceforth think of $C$ as a free module over this basis, so that
\[
C = \Span_{\F}\{x, y_i, z_i\} \otimes_{\F} \cR.
\]
Note that $\im \d_V$ is contained in $\Span_{\F[V]}\{z_i\}$. We will also have cause to consider the $\F[U]$-module $C/V$, which we identify with 
\[
C/V = \Span_{\F}\{x, y_i, z_i\} \otimes_{\F} \F[U],
\]
as well as the $\F$-vector space $C/(U,V)$, which we identify with
\[
C/(U,V) = \Span_{\F}\{x, y_i, z_i\}.
\]
Note that these identifications allow us to view elements of $C/(U, V)$ as elements of $C/V$ (and elements of $C/V$ as elements of $C$) in the obvious way -- an $\F$-linear combination of basis elements in $C/(U, V)$ may be viewed as the same linear combination in $C/V$, and so on. That is, they specify lifts from $C/(U,V)$ to $C/V$ and from $C/V$ to $C$.

Now let $P$ be the submodule of $C/V$ consisting of elements $w$ such that some $U$-power of $w$ lies in the image of $\d_U$:
\[
P = \{w \in C/V \co U^nw \in \im \d_U \text{ for some } n \geq 0\}.
\]
Note that $P$ has the property that if $Uw \in P$, then $w \in P$. Moreover, by the fact that $\d_U^2 = 0$, we have that every element $w \in P$ is a $\d_U$-cycle, that is, $\d_U w =0$. Choose an $\F[U]$-basis $p_1, \ldots, p_r$ for $P$. Let $\bar{p}_i$ denote the reduction of $p_i$ modulo $U$ in $C/(U, V)$. Explicitly, if $p_i$ is a linear combination (over $\F[U]$) of the basis elements $\{x, y_i, z_i\}$, then $\bar{p}_i$ consists of those terms which are not decorated by any powers of $U$. Note that $p_i$ differs from the canonical lift of $\bar{p}_i$ by an element in $\im U$.

We claim that the $\bar{p}_i$ are linearly independent as elements of $C/(U, V)$. Suppose not. Then we have some linear combination 
\[
\bar{p}_{i_1} + \cdots + \bar{p}_{i_k} = 0.
\]
Lifting this to $C/V$, this implies that $p_{i_1} + \cdots + p_{i_k} = Uw$ for some $w$. However, this means that $w \in P$. Writing $w$ as a linear combination of the $p_i$ gives a contradiction.

Consider the subspaces of $C/(U, V)$ given by $\bar{P} = \Span_{\F}\{\bar{p}_1, \ldots, \bar{p}_r\}$ and $\bar{Z} = \Span_{\F}\{z_i\}$. Extend the linearly independent set $\{\bar{p}_1, \ldots, \bar{p}_r\}$ to a basis
\[
\{\bar{p}_1, \ldots, \bar{p}_r,  z_{i_1}, \ldots, z_{i_s}\}
\]
for $\bar{P} + \bar{Z}$ in $C/(U, V)$. We claim that $x$ (viewed as an element of $C/(U, V)$) does not lie in $\bar{P} + \bar{Z}$. Indeed, if it did, we would have $x = \bar{p} + \sum z_{i_j}$ for some $\bar{p} \in \bar{P}$ and sum of the $z_{i_j}$. Lifting this to $C/V$ shows that
\[
x + \sum z_{i_j} + Uw \in P
\]
for some $w \in C/V$. By construction of $P$, we have that the above expression is a $\d_U$-cycle. Viewing it as an element of $C$, we also see that it is a $\d_V$-cycle, since $\d_V x = \d_V z_i = 0$ and $\d_V(Uw) = U \d_V w = 0$. This means that we can specify a local map from $\cR$ to $C$ by sending the generator of $\cR$ to $x + \sum z_{i_j} + Uw$, which generates the $V$-tower in $C/U$ (by definition of $x$ and the $z_i$). This would contradict the hypothesis of the lemma. Thus, $x \notin \bar{P} + \bar{Z}$.

Now consider the set of generators $S = \{x, p_1, \dots, p_r, z_{i_1}, \dots, z_{i_s}\}$ in $C/V$. It is straightforward to check that this is linearly independent by reducing any putative linear relation modulo $U$. We also claim that if  $Uw \in \Span_{\F[U]} S$, then $w \in \Span_{\F[U]} S$. Indeed, suppose not. Then we have
\[
U w = U^*x + \sum U^* p_i + \sum U^* z_{i_j}
\]
where at least one term on the right-hand side appears with a $U$-exponent of zero. Reducing both sides modulo $U$, we obtain a nontrivial linear relation among the generators $\{x, \bar{p}_1, \ldots, \bar{p}_r,  z_{i_1}, \ldots, z_{i_s}\}$, a contradiction. It follows that we may extend $S$ to an $\F[U]$-basis
\[
\{x, p_1, \dots, p_r, z_{i_1}, \dots, z_{i_s}, w_1, \dots, w_t\}
\]
for all of $C/V$.\footnote{As in the proof of Lemma~~\ref{lem:reduced}, we are using the fact that if $N$ is a (free) submodule of a free module $M$, then a basis for $N$ can be extended to a basis for $M$ if and only if $M/N$ is also free.} This then gives an $\cR$-basis for all of $C$.

By construction, 
\[ D = \Span_\cR \{ p_1, \dots, p_r, z_{i_1}, \dots, z_{i_s}, w_1, \dots, w_t\}\]
is a subcomplex of $C$. Indeed, the image of $\d_U$ is contained in the span of the $p_i$. Similarly, the image of $\d_V$ is contained in the span of the $p_i$ and $z_{i_j}$. To see this, note that any $z_k$ is an $\F$-linear combination of the $\bar{p}_i$ and the $z_{i_j}$. Hence (viewing these as elements of $C$), we have
\[
z_k = \sum p_i + \sum z_{i_j} + Uw
\]
for some element $w$, since $\bar{p}_i = p_i \mod U$. Thus for any $y_k$, we have 
\[
\d_V y_k = V^{\eta_k} z_k = V^{\eta_k}\left( \sum p_i + \sum z_{i_j} + Uw \right) = V^{\eta_k} \left( \sum p_i + \sum z_{i_j} \right).
\]
Hence $\d D \subset D$. Then the quotient map
\[ C \to C / D \cong \cR \]
is a local map from $C$ to $\cR$.
\end{proof}

\begin{proof}[Proof of Proposition \ref{prop:totalorder}]
We need to show totality of $\leq$. Let $C_1$ and $C_2$ be two knot-like complexes. Consider $C_1 \otimes C_2^\vee$. By Lemma \ref{lem:totalorder}, we have that either $C_1 \otimes C_2^\vee \geq \cR$ or  $C_1 \otimes C_2^\vee \leq \cR$. By tensoring with $C_2$, we have that either $C_1 \geq C_2$ or $C_1 \leq C_2$, as desired.
\end{proof}

\begin{remark}
The group $\KL$ should be compared to to the group $\cCFK$ defined in \cite{Hominfiniterank} using $\varep$-equivalence. Indeed, $\cCFK$ is isomorphic (as an ordered group) to the subgroup of $\KL$ generated by $\{ \CFKUV(K) \mid K \text{ a knot in } S^3\}$. In particular, the order $\leq$ defined in Definition \ref{def:le} agrees with the order given by $\varep$.
\end{remark}

\section{Standard complexes and their properties}\label{sec:standards}
In this section, we define a convenient family of knot-like complexes called standard complexes.

\begin{remark}
The reader should compare with \cite[Section 4]{DHSThomcobord}, which carries out the analogous construction in the setting of almost $\iota$-complexes. Indeed, an almost $\iota$-complex may be viewed as a complex over the ring $\F[U, Q]/(Q^2=QU=0)$. In our case, this corresponds (roughly) to passing to the ring $\F[U, V]/(UV = V^2 = 0)$.
\end{remark}

\subsection{Standard complexes}\label{subsec:standards}
Let $C$ be a knot-like complex generated by $x_0, \dots, x_n$. We say there is a \emph{$U^m$-arrow} between $x_i$ and $x_j$ for  $m \in \N$ if one of the following occurs:
\begin{enumerate}
	\item \label{it:itoj} $\d_U x_i = U^m x_j$, or
	\item \label{it:jtoi} $\d_U x_j = U^m x_i$.	
\end{enumerate}
The arrow goes from $x_i$ to $x_j$ in \eqref{it:itoj} and from $x_j$ to $x_i$ in \eqref{it:jtoi}. We define \emph{$V^m$-arrows} analogously by replacing $U$ with $V$.

\begin{remark}
In the traditional depiction of $\CFKi$ as a bifiltered complex in the $(i, j)$-plane, each generator (over $\mathbb{F}$) is placed in its appropriate bigrading and is decorated with a power of $U$. An arrow between two generators indicates that one (with its $U$-power decoration) appears in the differential of the other. This is not quite the same as the pictoral description we use here. Instead, we suppress writing the decorations of our generators and use their spatial placement in the plane to determine the appropriate $U$ or $V$ powers appearing in the differential. That is, a horizontal arrow of length $m$ from $x_k$ to $x_\ell$ indicates that $x_\ell$ appears in the differential of $x_k$ with a coefficient of $U^m$, and similarly for vertical arrows and powers of $V$. It can be shown, however, that (modulo an infinite number of translations) this produces the same shape as in the previous picture.
\end{remark}

\begin{definition}\label{def:standard}
Let $n \in 2\N$, and let $(b_1, \dots, b_n)$ be a sequence of nonzero integers. A \emph{standard complex of type $(b_1, \dots, b_n)$}, denoted by $C(b_1, \dots, b_n)$, is the knot-like complex freely generated over $\cR$ by
\[ \{x_0, x_1, \dots, x_n\}. \]
Each pair of generators $x_i$ and $x_{i-1}$ for $i$ odd are connected by $U^{|b_{i}|}$-arrows, and each pair of generators $x_i$ and $x_{i-1}$ for $i$ even are connected by $V^{|b_{i}|}$-arrows. The direction is determined by the sign of $b_{i}$, as follows. If $b_{i}$ is positive, then the arrow goes from $x_{i}$ to $x_{i-1}$, and if $b_{i}$ is negative, then the arrow goes from $x_{i-1}$ to $x_{i}$. We call $n$ the \emph{length} of the standard complex and $\{x_i\}_{i = 1}^n$ the \emph{preferred basis}. 
Explicitly, the differential on $C(b_1, \dots, b_n)$ is as follows. For $i$ odd,
\begin{align*}
&\d_U x_{i-1} = U^{|b_i|} x_{i}  \quad \text{ if } b_i < 0 \\
&\d_U x_i = U^{b_i} x_{i-1} \ \ \quad \text{ if } b_i > 0 
\end{align*}
while for $i$ even,
\begin{align*}
&\d_V x_{i-1} = V^{|b_i|} x_{i} \quad \text{ if } b_i < 0 \\
&\d_V x_{i} = V^{b_i} x_{i-1} \ \ \quad \text{ if } b_i > 0. 
\end{align*}
All other differentials are zero. 

Note that $x_0$ generates $H_*(C(b_1, \dots, b_n)/U)/V\text{-torsion}$. Similarly, $x_n$ generates $H_*(C(b_1, \dots, b_n)/V)/U\text{-torsion}$. There is thus a unique grading on $C(b_1, \dots, b_n)$ which makes it into a knot-like complex: namely, $\gr_U(x_0) = 0$ and $\gr_V(x_n)=0$. The fact that the differential has degree $(-1, -1)$ then determines the rest of the gradings. Note that $\gr_U(x_i) \equiv \gr_V(x_i) \equiv i \mod 2$; we refer to this as the \textit{parity} of (the grading of) a generator of $C(b_1, \dots, b_n)$. 
\end{definition}

\begin{definition}
We say a standard complex $C(b_1, \dots, b_n)$ is \emph{symmetric} if $b_i = -b_{n+1-i}$.
\end{definition}

\begin{example}
We define the trivial standard complex $C(0) \cong \cR$ to be the complex generated over $\cR$ by a single element with $U$- and $V$-grading zero.
\end{example}

\begin{example}
The standard complex $C(1, -2, 2, -1)$ is generated over $\cR$ by 
\[ x_0, x_1, x_2, x_3, x_4 \]
with 
\[ \d x_0 = \d x_2 = \d x_4 =0, \qquad \d x_1 = Ux_0 + V^2 x_2, \qquad \d x_3 = U^2 x_2 + V x_4. \] 
The gradings of the generators are
\begin{align*}
	\gr(x_0) &= (0, -6) \\
	\gr(x_1) &= (-1, -5) \\
	 \gr(x_2) &= (-2, -2) \\
	 \gr(x_3) &= (-5, -1) \\
	 \gr(x_4) &= (-6, 0).
 \end{align*}
See Figure \ref{fig:1221} for a visual depiction of $C(1, -2, 2, -1)$, where a horizontal (resp. vertical) arrow of length $m$ from $x_i$ to $x_j$ represents a $U^m$-arrow (resp. $V^m$-arrow). Note that to read off the standard complex from the figure, we start at $x_0$ and follow the unique path to $x_4$, recording the direction and length of each arrow that we traverse. Namely, traversing an arrow of length $m$ against the direction of the arrow yields a $+m$, while traversing an arrow of length $m$ in the direction of the arrow yields a $-m$.
\end{example}

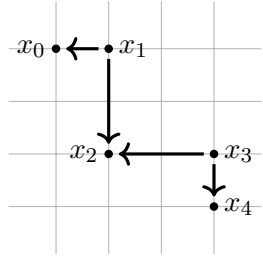
\begin{figure}[ht]
\centering
\begin{tikzpicture}[scale=0.7]
	\draw[step=1, black!30!white, very thin] (-0.9, -0.9) grid (3.9, 3.9);
	\filldraw (0, 3) circle (2pt) node[] (a){};
	\filldraw (1, 3) circle (2pt) node[] (b){};
	\filldraw (1, 1) circle (2pt) node[] (c){};
	\filldraw (3, 1) circle (2pt) node[] (d){};
	\filldraw (3, 0) circle (2pt) node[] (e){};
	\draw [very thick, <-] (a) -- (b);
	\draw [very thick, <-] (c) -- (b);
	\draw [very thick, <-] (c) -- (d);
	\draw [very thick, <-] (e) -- (d);
	\node[left] at (a){$x_0$};
	\node[right] at (b){$x_1$};
	\node[left] at (c){$x_2$};
	\node[right] at (d){$x_3$};
	\node[right] at (e){$x_4$};
\end{tikzpicture}
\caption{The standard complex $C(1, -2, 2, -1)$. A horizontal (respectively, vertical) arrow of length $m$ from $x_i$ to $x_j$ means that $\d_U x_i = U^m x_j$ (respectively, $\d_V x_i = V^m x_j$).}
\label{fig:1221}
\end{figure}

\begin{example}
The standard complex $C(-1, 1)$ is generated over $\cR$ by 
\[ x_0, x_1, x_2 \]
with 
\[ \d x_0 = Ux_1, \qquad \d x_1 = 0, \qquad  \d x_2 = V x_1 \] 
and gradings
\begin{align*}
	\gr(x_0) &= (0, 2) \\
	\gr(x_1) &= (1, 1) \\
	 \gr(x_2) &= (2, 0).
\end{align*}
See Figure \ref{fig:11} for a visual depiction.
\end{example}

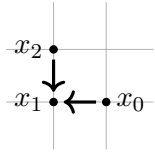
\begin{figure}[ht]
\centering
\begin{tikzpicture}[scale=0.7]
	\draw[step=1, black!30!white, very thin] (-0.9, -0.9) grid (1.9, 1.9);
	\filldraw (1, 0) circle (2pt) node[] (a){};
	\filldraw (0, 0) circle (2pt) node[] (b){};
	\filldraw (0, 1) circle (2pt) node[] (c){};
	\draw [very thick, ->] (a) -- (b);
	\draw [very thick, ->] (c) -- (b);
	\node[right] at (a){$x_0$};
	\node[left] at (b){$x_1$};
	\node[left] at (c){$x_2$};
\end{tikzpicture}
\caption{The standard complex $C(-1, 1)$.}
\label{fig:11}
\end{figure}

\begin{example}\label{ex:121121}
The standard complex $C(1, -2, -1, 1, 2, -1)$ is generated over $\cR$ by 
\[ x_0, x_1, x_2, x_3, x_4, x_5, x_6 \]
with nonzero differentials
\[ \d x_1 = Ux_0 + V^2 x_2, \quad \d x_2 = Ux_3, \quad \d x_4 = V x_3, \quad \d x_5 = U^2 x_4 + V x_6 \]
with gradings
\begin{align*}
	\gr(x_0) &= (0, -4) \\
	\gr(x_1) &= (-1, -3) \\
	 \gr(x_2) &= (-4, -2) \\
	 \gr(x_3) &= (-3, -3) \\
	 \gr(x_4) &= (-2, -4) \\
	 \gr(x_5) &= (-3, -1) \\
	 \gr(x_6) &= (-4, 0).
 \end{align*}
See Figure \ref{fig:121121} for a visual depiction.
\end{example}

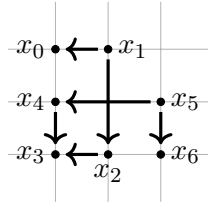
\begin{figure}[ht]
\centering
\begin{tikzpicture}[scale=0.7]
	\draw[step=1, black!30!white, very thin] (-0.9, -0.9) grid (2.9, 2.9);
	\filldraw (0, 2) circle (2pt) node[] (a){};
	\filldraw (1, 2) circle (2pt) node[] (b){};
	\filldraw (1, 0) circle (2pt) node[] (c){};
	\filldraw (0, 0) circle (2pt) node[] (d){};
	\filldraw (0, 1) circle (2pt) node[] (e){};
	\filldraw (2, 1) circle (2pt) node[] (f){};
	\filldraw (2, 0) circle (2pt) node[] (g){};
	\draw [very thick, <-] (a) -- (b);
	\draw [very thick, <-] (c) -- (b);
	\draw [very thick, ->] (c) -- (d);
	\draw [very thick, ->] (e) -- (d);
	\draw [very thick, <-] (e) -- (f);
	\draw [very thick, ->] (f) -- (g);
	\node[left] at (a){$x_0$};
	\node[right] at (b){$x_1$};
	\node[below] at (c){$x_2$};
	\node[left] at (d){$x_3$};
	\node[left] at (e){$x_4$};
	\node[right] at (f){$x_5$};
	\node[right] at (g){$x_6$};
\end{tikzpicture}
\caption{The standard complex $C(1, -2, -1, 1, 2, -1)$. }
\label{fig:121121}
\end{figure}

\begin{lemma}\label{lem:dual}
The dual of $C(b_1, \dots, b_n)$ is $C(-b_1, \dots, -b_n)$.
\end{lemma}

\begin{proof}
This is a straightforward consequence of the definitions.
\end{proof}

\subsection{An unusual order on the integers}
Let $\Zbang = (\Z, \leqbang)$ denote the integers with the following unusual order:
\[ -1 \lebang -2 \lebang -3 \lebang \dots \lebang  0 \lebang \dots \lebang 3 \lebang 2 \lebang 1. \]
We will see shortly the utility of this strange order. Note that for $a, b \neq 0$, we have $a \lebang b$ if and only if $\smash{\frac{1}{a} < \frac{1}{b}}$, where $<$ denotes the usual order on $\Q$. Since $a \gebang 0$ if and only if $a > 0$, the sign of $a \in \Zbang$ coincides with the usual definition (that is, $a$ is positive if $a>0$ and negative if $a<0$).

\subsection{Ordering standard complexes}

We consider $\Zbang$-valued sequences, with the lexicographic order induced by $\leqbang$. We take the convention that in order to compare two sequences of different lengths, we append sufficiently many trailing zeros to the shorter sequence so that the sequences have the same length.

\begin{proposition}\label{prop:lex}
Standard complexes are ordered lexicographically as $\Zbang$-valued sequences with respect to the total order on $\KL$.
\end{proposition}

The proof of Proposition~\ref{prop:lex} consists of a number of straightforward but technical verifications regarding local maps between standard complexes. We have included the details so that the reader will become accustomed to routine manipulations involving these definitions.

\begin{lemma}\label{lem:leqlex}
Let $(a_1, \dots, a_m) \leqbang (b_1, \dots, b_n)$ in the lexicographic order on $\Zbang$-valued sequences. Then $C(a_1, \dots, a_m) \leq C(b_1, \dots, b_n)$ in $\KL$.
\end{lemma}

\begin{proof}
If $(a_1, \dots, a_m) = (b_1, \dots, b_n)$, then it is clear that the complexes in question are locally equivalent by taking the obvious identity map. Thus, assume that $(a_1, \dots, a_m) < (b_1, \dots, b_n)$. Suppose that the two sequences agree up to index $k$, so that $a_i = b_i$ for $1 \leq i < k$ and $a_k \lebang b_k$. 

Let $\{x_i\}$ and $\{y_i\}$ be the preferred bases for $C(a_1, \dots, a_m)$ and $C(b_1, \dots, b_n)$, respectively. Define 
\[ f \co C(a_1, \dots, a_m) \to C(b_1, \dots, b_n) \]
by
\begin{align*}
	f(x_i) = \begin{cases}
		y_i &\text{ if } 0 \leq i < k \\
		0 &\text{ if } i > k.
	\end{cases}
\end{align*}
In order to define $f(x_k)$, we proceed with some elementary casework based on the value of $k$. First, suppose that $k \leq \min \{ m, n\}$, and consider the parity of $k$:
\begin{enumerate}
	\item If $k$ is odd:
	\begin{enumerate}
		\item If $a_k \lebang b_k < 0$, then let $f(x_k) = U^{a_k - b_k} y_k$. It is straightforward to verify that $f$ is a chain map; the only nontrivial checks are that $\d_U f(x_{k-1}) = f \d_U (x_{k-1})$ and $\d f(x_{k}) = f \d (x_{k})$. To verify the former, we see that
		\begin{align*}
			\d_U f(x_{k-1}) = \d_U y_{k-1} = U^{-b_k} y_k,
		\end{align*}
		while 
		\begin{align*}
			f \d_U (x_{k-1}) = f ( U^{-a_k} x_k) = U^{-a_k} U^{a_k - b_k} y_k.
		\end{align*}
To verify the latter, we see that
		\begin{align*}
			\d f(x_{k}) = \d U^{a_k - b_k} y_k = U^{a_k - b_k} \d_V y_k.
		\end{align*}
		This is zero, since either $\d_V y_k = 0$ or $\d_V y_k = V^{-b_{k+1}}y_{k+1}$ and $UV =0$. Meanwhile, $f \d (x_{k}) = 0$ since $\d x_k$ is either equal to zero or $V^{-a_{k+1}}x_{k+1}$.
		\item If $a_k < 0 < b_k$, then let $f(x_k) = 0$. It is straightforward to verify $f$ is a chain map; the only nontrivial check is that $\d_U f(x_{k-1}) = f \d_U (x_{k-1})$. This follows from the fact that $b_k > 0$ (i.e., $\d_U y_{k-1} = 0$).
		\item If $0 < a_k \lebang b_k$, then let $f(x_k) = U^{a_k - b_k} y_k$. It is straightforward to verify that $f$ is a chain map; the only nontrivial check is that $\d f(x_{k}) = f \d (x_{k})$. This follows from the fact that 
		\begin{align*}
			\d f(x_{k}) &= \d U^{a_k - b_k} y_k \\
					&= U^{a_k - b_k} (\d_U y_k + \d_V y_k) \\
					&= U^{a_k - b_k} U^{b_k} y_{k-1},
		\end{align*}
		while
		\begin{align*}
			f \d (x_{k}) = f U^{a_k} x_{k-1} = U^{a_k} y_{k-1}.
		\end{align*}
	\end{enumerate}
	\item The case when $k$ even is similar, but with $V$ playing the role of $U$.
\end{enumerate}
\noindent
Now assume that $k > \min \{ m ,n \}$. We consider the following two cases:
\begin{enumerate}
	\item Suppose that $n > m$. Then $k = m+1$, and 
	\[
	(b_i)_{i=1}^n = (a_1, \dots, a_m, b_{m+1}, \dots, b_n)
	\]
	with $b_{m+1} > 0$. Let $f$ be the obvious inclusion map. As above, it is easily checked that $f$ commutes with $\d$.
	\item Suppose that $m > n$. Then $k = n+1$, and 
	\[
	(a_i)_{i=1}^m = (b_1, \dots, b_n, a_{n+1}, \dots, a_m)
	\]
	with $a_{n+1} < 0$. Let $f$ be the obvious projection map. As above, it is easily checked that $f$ commutes with $\d$.
\end{enumerate}
It is clear that $f$ is local, since $f(x_0) = y_0$. This completes the proof.
\end{proof}

\begin{lemma}\label{lem:localsupport}
Let $C_1= C(a_1, \dots, a_m)$ and $C_2 = C(b_1, \dots, b_n)$ be standard complexes with preferred bases $\{x_i\}$ and $\{ y_i\}$, respectively. Suppose that $a_i = b_i$ for all $1 \leq i \leq k$ and that $f \co C_1 \to C_2$ is a local map. Then $f(x_i)$ is supported by $y_i$ for all $0 \leq i \leq k$.
\end{lemma}

\begin{proof}
We proceed by induction on $i$. The base case $i = 0$ follows from the fact that $f$ is local. Thus, let $i < k$, and assume that $f(x_i)$ is supported by $y_i$. We show that $f(x_{i+1})$ is supported by $y_{i+1}$. Suppose that $i$ is even. We consider the following two cases:
\begin{enumerate}
	\item Suppose that $a_{i+1} = b_{i+1} < 0$. Then $\d_U f(x_i) = f \d_U (x_i) = U^{|a_{i+1}|} f ( x_{i+1})$. By the induction hypothesis, $f(x_i)$ is supported by $y_i$. We have that $\d_U y_i = U^{|b_{i+1}|} y_{i+1}$ and that $y_i$ is the unique element in $C_2$ such that $\d_U$ of it is supported by a $U$-power of $y_{i+1}$. It follows that $f(x_{i+1})$ must be supported by $y_{i+1}$.
	\item Suppose that $a_{i+1} = b_{i+1} > 0$. Then $\d_U f(x_{i+1}) = f \d_U (x_{i+1}) = U^{|a_{i+1}|} f ( x_{i})$. By the induction hypothesis, $f(x_i)$ is supported by $y_i$. We have that $\d_U y_{i+1} = U^{|b_{i+1}|} y_i$ and that $y_{i+1}$ is the unique basis element in $C_2$ such that $\d_U$ of it is supported by a $U$-power of $y_i$. It follows that $f(x_{i+1})$ must be supported by $y_{i+1}$.
\end{enumerate}
The case $i$ odd is similar, but with $V$ playing the role of $U$.
\end{proof}

\begin{lemma}\label{lem:>lex}
Let $(a_1, \dots, a_m) \gebang (b_1, \dots, b_n)$ in the lexicographic order on $\Zbang$-valued sequences. Then there is no local map from $C_1 = C(a_1, \dots, a_m)$ to $C_2 = C(b_1, \dots, b_n)$.
\end{lemma}

\begin{proof}
Suppose that $a_i = b_i$ for $i < k$ and that $a_k \gebang b_k$. We proceed by contradiction. Assume there is a local map $f \co C_1 \to C_2$. We begin by considering the case when $k \leq \min \{ m, n\}$:
\begin{enumerate}
	\item Suppose that $k$ is odd. We have three further subcases:
	\begin{enumerate}
		\item Suppose that $b_k \lebang a_k < 0$. Then $\d_U x_{k-1} = U^{|a_k|} x_k$ and $\d_U y_{k-1} = U^{|b_k|} y_k$. Furthermore, $y_{k-1}$ is the unique basis element of $C_2$ such that $\d_U$ of it is supported by a $U$-power of $y_k$. By Lemma \ref{lem:localsupport}, $f(x_{k-1})$ is supported by $y_{k-1}$. It follows that $f \d_U (x_{k-1}) = \d_U f(x_{k-1})$ is supported by $U^{|b_k|} y_k$. Hence $f(U^{|a_k|} x_k)$ must be supported by $U^{|b_k|} y_k$, which is a contradiction, since $b_k \lebang a_k < 0$, i.e., $|b_k| < |a_k|$ where $<$ denotes the usual ordering on $\Z$.
		\item Suppose that $b_k < 0 < a_k$. Then $\d_U y_{k-1} = U^{|b_k|} y_k$ and $\d_U x_{k-1} = 0$. Furthermore, $y_{k-1}$ is the unique basis element in $C_2$ such that $\d_U$ of it is supported by a $U$-power of $y_k$. By Lemma \ref{lem:localsupport}, $f(x_{k-1})$ is supported by $y_{k-1}$. But $0 = f \d_U(x_{k-1}) = \d_U f(x_{k-1})$, a contradiction, since the right-hand side is supported by $U^{|b_k|} y_k$.
		\item Suppose that $0 < b_k \lebang a_k$. Then $\d_U x_k = U^{a_k} x_{k-1}$ and $\d_U y_k = U^{b_k} y_{k-1}$. Furthermore, $y_k$ is the unique basis element in $C_2$ such that $\d_U$ of it is supported by a $U$-power of $y_{k-1}$. By Lemma \ref{lem:localsupport}, $f(x_{k-1})$ is supported by $y_{k-1}$. Then  $\d_U f(x_k) = f \d_U (x_k) = f(U^{a_k} x_{k-1})$, where the right-hand side is supported by $U^{a_k} y_{k-1}$. Hence $f(x_k)$ must be supported by $U^{a_k-b_k}y_k$, a contradiction since $0 < b_k \lebang a_k $, i.e., $b_k > a_k$ where $<$ denotes the usual ordering on $\Z$.
	\end{enumerate}
	\item The case when $k$ is even is similar, but with $V$ playing the role of $U$.
\end{enumerate}
Now assume that $k > \min \{ m, n\}$. We consider the following two cases:
\begin{enumerate}
	\item Suppose $n > m$. Then $k = m+1$ and $(b_i)_{i=1}^n = (a_1, \dots, a_m, b_{m+1}, \dots, b_n)$. Then $b_{m+1} < 0$, that is, $\d_U y_m = U^{|b_{m+1}|} y_{m+1}$ and $y_m$ is the unique element in $C_2$ such that $\d_U$ of it is supported by a $U$-power of $y_{m+1}$. By Lemma \ref{lem:localsupport}, $f(x_m)$ is supported by $y_m$. But $0 = f \d_U (x_m) = \d_U f(x_m) \neq 0$ since $\d_U f(x_m)$ is supported by $U^{|b_{m+1}|} y_{m+1}$.
	\item Suppose $m > n$. Then $k = n+1$ and $(a_i)_{i=1}^m = (b_1, \dots, b_n, a_{n+1}, \dots, a_m)$. Then $a_{n+1} > 0$, that is, $\d_U x_{n+1} = U^{|a_{n+1}|} x_{n}$. Furthermore, no $U$-power of $y_n$ appears as $\d_U$ of any element in $C_2$. By Lemma \ref{lem:localsupport}, $f(x_n)$ is supported by $y_n$. But $\d_U f (x_{n+1}) = f \d_U ( x_{n+1}) = f ( U^{|a_{n+1}|} x_{n} )$ is supported by $U^{|a_{n+1}|} y_{n}$, a contradiction.
\end{enumerate}
This completes the proof.
\end{proof}

\begin{proof}[Proof of Proposition \ref{prop:lex}]
The proposition follows immediately from Lemmas \ref{lem:leqlex} and \ref{lem:>lex}.
\end{proof}

\subsection{Semistandard complexes}
In future sections, we will also find it useful to have the following generalization of standard complexes.

\begin{definition}
Let $n \in 2\N - 1$, and let $(b_1, \dots, b_n)$ be a sequence of nonzero integers. The \emph{semistandard complex} $C'(b_1, \dots, b_n)$ is the subcomplex of the standard complex $C(b_1, \dots, b_n, 1)$ generated by $x_0, x_1, \dots, x_n$. We call these the \emph{preferred generators} of $C'(b_1, \dots, b_n)$. (The choice $b_{n+1}=1$ here is unimportant; any $b_{n+1} > 0$ is allowed.) 

We stress that a semistandard complex is \emph{not} a knot-like complex; indeed, for $C'$ a semistandard complex, $H_*(C'/U)/V\text{-torsion}$ has two $V$-towers, which are generated by $x_0$ and $x_n$. Note that since $n$ is odd, the gradings of $x_0$ and $x_n$ have opposite parities.
\end{definition}

We use the symbol $'$ to distinguish semistandard complexes from standard complexes; that is, $C'(b_1, \dots, b_n)$ denotes a semistandard complex (where $n$ is odd) while $C(b_1, \dots, b_n)$ denotes a standard complex (where $n$ is even).

\begin{definition}
A grading-preserving $\cR$-equivariant chain map
\[ f \co C'(b_1, \dots, b_n) \to C \]
from a semistandard complex to a knot-like complex $C$ is said to be \emph{local} if the class of $f(x_0)$ generates $H_*(C/U)/V\text{-torsion}$.
\end{definition}

\begin{example}
The semistandard complex $C'(1, -2, -1, 1, 2)$ is generated over $\cR$ by 
\[ x_0, x_1, x_2, x_3, x_4, x_5 \]
with nonzero differentials
\[ \d x_1 = Ux_0 + V^2 x_2, \quad \d x_2 = Ux_3, \quad \d x_4 = V x_3, \quad \d x_5 = U^2 x_4. \]
See Figure \ref{fig:12112} for a visual depiction. 
\end{example}

\begin{figure}[ht]
\centering
\begin{tikzpicture}[scale=0.7]
	\draw[step=1, black!30!white, very thin] (-0.9, -0.9) grid (2.9, 2.9);
	\filldraw (0, 2) circle (2pt) node[] (a){};
	\filldraw (1, 2) circle (2pt) node[] (b){};
	\filldraw (1, 0) circle (2pt) node[] (c){};
	\filldraw (0, 0) circle (2pt) node[] (d){};
	\filldraw (0, 1) circle (2pt) node[] (e){};
	\filldraw (2, 1) circle (2pt) node[] (f){};
	\draw [very thick, <-] (a) -- (b);
	\draw [very thick, <-] (c) -- (b);
	\draw [very thick, ->] (c) -- (d);
	\draw [very thick, ->] (e) -- (d);
	\draw [very thick, <-] (e) -- (f);
	\node[left] at (a){$x_0$};
	\node[right] at (b){$x_1$};
	\node[below] at (c){$x_2$};
	\node[left] at (d){$x_3$};
	\node[left] at (e){$x_4$};
	\node[right] at (f){$x_5$};
\end{tikzpicture}
\caption{The semistandard complex $C'(1, -2, -1, 1, 2)$. }
\label{fig:12112}
\end{figure}
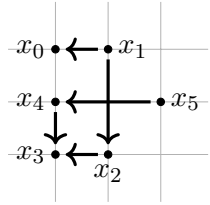

\subsection{Short maps}
It will often be useful for us to consider module maps from a standard complex $C(b_1, \dots, b_n)$ to a knot-like complex $C$ that are chain maps except possibly at $x_n$. We make this notion precise with the following definition.

\begin{definition}
Let $C_1= C(b_1, \dots, b_n)$ be a standard complex and $C_2$ a knot-like complex. An absolutely $U$-graded, relatively $V$-graded module map $f \co C_1 \to C_2$ is called a \emph{short map}, denoted
\[ f \co C(b_1, \dots, b_n) \leadsto C_2, \]
if $f \d (x_i) + \d f (x_i) = 0$ for $1 \leq i \leq n-1$ and $f \d_V(x_n) + \d_V f(x_n) = 0$. If $f$ induces an isomorphism on $H_*( C_i  /U)/V \text{-torsion}$, then we call $f$ a \emph{short local map}.
\end{definition}
\noindent
We similarly define short maps for semistandard complexes:

\begin{definition}
Let $C_1 = C'(b_1, \dots, b_n)$ be a semistandard complex and $C_2$ a knot-like complex. An absolutely $U$-graded, relatively $V$-graded module map $f \co C_1 \to C_2$ is called a \emph{short map}, denoted
\[ f \co C'(b_1, \dots, b_n) \leadsto C, \]
if $f \d (x_i) + \d f (x_i) = 0$ for $1 \leq i \leq n-1$ and $f \d_U(x_n) + \d_U f(x_n) = 0$. If the class of $f(x_0)$ generates $H_*(C_2/U)/V\text{-torsion}$, then we call $f$ a \emph{short local map}.
\end{definition}

The following lemma states that given a short map, we can extend it to an actual chain map (from a different domain).

\begin{lemma}[Extension Lemma]\label{lem:shortextension}
Let
\[ f \co C(b_1, \dots, b_n) \leadsto C \]
be a short map from a standard complex to $C$. Then there exists an $\cR$-equivariant chain map
\[ g \co C(b_1, \dots, b_n, b_{n+1}, \dots, b_m) \to C \]
for some $b_i$, $n+1 \leq i \leq m$ such that $f$ and $g$ agree on the generators of $C(b_1, \dots, b_n)$ (viewed as generators of $C(b_1, \dots, b_n, b_{n+1}, \dots, b_m)$ in the obvious way). Moreover, if $f$ is local, then $g$ is local.
\end{lemma}

\begin{proof}
Consider $f(x_n)$. If $\d_U f(x_n) = 0$, then $f$ is already a chain map and we are done. Thus, suppose that $\d_U f(x_n) = U^c z$ for some $z \in C$, $c \geq 1$. Define a short map $f' \co C'(b_1, \dots, b_n, -1) \leadsto C$ by setting $f'(x_i) = f(x_i)$ for $0 \leq i \leq n$ and $f'(x_{n+1}) = U^{c-1} z$. We now consider several cases:

\begin{enumerate}
\item If $c> 1$, then extend the domain of $f'$ to $C(b_1, \dots, b_n, -1, -1)$  by setting $f'(x_{n+2})=0$. It is easily checked that $f'$ then provides the desired $\cR$-equivariant chain map.
\item If $c=1$ and $\d_V z = 0$, then we may again extend the domain of $f'$ to $C'(b_1, \dots, b_n, -1, -1)$ by setting $f'(x_{n+2})=0$. It is easily checked that $f'$ then provides the desired $\cR$-equivariant chain map.
\item If $c=1$ and $\d_V z= V^d w$ for some $w \in C$, $d \geq 1$, then we proceed as in the beginning of the proof, except replacing the role of $U$ with $V$. That is, extend the short map
\[ f' \co C'(b_1, \dots, b_n, -1) \leadsto C \]
to a short map
\[ f'' \co C(b_1, \dots, b_n, -1, -1) \leadsto C. \]
Iterate this procedure. Note that both the $U$- and $V$-gradings of the final preferred generator of $C(b_1, \dots, b_n, -1, -1, \dots, -1, -1)$ increase as the length of standard complex increases. Since $C$ is finitely generated, the gradings of its generators are bounded above. Hence it is easily checked that at some point this process must terminate, yielding the desired extension.
\end{enumerate}

Since $g(x_0) = f(x_0)$, it is clear that $g$ is local if $f$ is local.
\end{proof}
\noindent
The analogous result holds for semistandard complexes:

\begin{lemma}\label{lem:shortextensionsemi}
Let
\[ f \co C'(b_1, \dots, b_n) \leadsto C \]
be a short map from a semistandard complex to $C$. Then there exists a $\cR$-equivariant chain map
\[ g \co C(b_1, \dots, b_n, b_{n+1}, \dots, b_m) \to C \]
for some $b_i$, $n+1 \leq i \leq m$  such that  $f$ and $g$ agree on the generators of $C'(b_1, \dots, b_n)$. Moreover, if $f$ is local, then $g$ is local.
\end{lemma}

\begin{proof}
Analogous to the proof of Lemma \ref{lem:shortextension}.
\end{proof}

\section{Numerical invariants $a_i$}\label{sec:ai}
In this section, we define a sequence of numerical invariants $(a_i)$ for any knot-like complex $C$, analogous to those constructed in \cite[Section 6]{DHSThomcobord}. Up to sign, these are the same as the invariants defined in \cite[Section 3]{Hominfiniterank}, which are also denoted by $(a_i)$. In the next section, we will see that the $a_i$ compute successive parameters in the standard complex representative of $C$.

Let $C$ be a knot-like complex. Define
\[ a_1(C) = \sup\bang \{ b_1 \in \Zbang \mid C(b_1, \dots, b_n) \leq C \}. \]
Here, $\sup\bang$ denotes the supremum  taken with respect to the (unusual!) order on $\Zbang$. We define $a_k(C)$ for $k \geq 2$ inductively, as follows. Suppose that we have already defined $a_i = a_i(C)$ for $1 \leq i \leq k$. If $a_k =0$, define $a_{k+1}(C) = 0$. Otherwise, define
\[ a_{k+1}(C) = \sup\bang \{ b_{k+1} \in \Zbang \mid C(a_1, \dots, a_k, b_{k+1}, \dots b_n) \leq C \}. \]
\noindent
That is, we consider the set of standard complexes $\leq C$ whose first $k$ symbols agree with the previously defined $a_i$. We then take the supremum over the family of $(k+1)$st symbols appearing in this set.\footnote{It will be implicit in the proof of Proposition~\ref{prop:sup} that this set of standard complexes is nonempty. More precisely, if $a_1, \ldots, a_k$ are all defined and nonzero, then there exists a standard complex of the form $C(a_1, \dots, a_k, b_{k+1}, \dots b_n)$ which is $\leq C$.} 

It will be convenient for us to have the following terminology.

\begin{definition}
Let $C$ be a knot-like complex, and let $n$ be a positive integer. Let $(a_1, \dots, a_n)$ be the sequence given by the first $n$ invariants $a_i = a_i(C)$, $1 \leq i \leq n$. We say that $(a_1, \dots, a_n)$ -- and, similarly, the standard complex $C(a_1, \dots, a_n)$ -- is \textit{$n$-maximal with respect to $C$}. Here, we identify $C(a_1, \dots, a_n, 0, \dots, 0) = C(a_1, \dots, a_n)$.
\end{definition}

The following proposition (combined with the extension lemma) shows that the supremum in the definition of $a_i$ is always realized.

\begin{proposition}\label{prop:sup}
Let $a_i = a_i(C)$. For each $n \in \N$, there is a short local map 
\[ f \co C(a_1, \dots, a_n) \leadsto C. \]
Here, we identify $C(a_1, \dots, a_n, 0, \dots, 0) = C(a_1, \dots, a_n)$.
\end{proposition}

\noindent
This is a consequence of the following lemmas.

\begin{lemma}\label{lem:semitostandard}
Let
\[ f \co C'(b_1, \dots, b_n) \to C \]
be a local map from a semistandard complex to knot-like complex $C$. Then there is some $b_{n+1} > 0$ such that we have a short local map from the standard complex $C(b_1, \dots, b_n, b_{n+1})$ to $C$:
\[ g \co C(b_1, \dots, b_n, b_{n+1}) \leadsto C. \]
\end{lemma}

\begin{proof}
Let $C' =  C'(b_1, \dots, b_n)$. Since $x_n$ is a cycle in $C'/U$, we have that $f(x_n)$ is a cycle in $C/U$. Moreover, the class of $f(x_n)$ must be $V$-torsion in $C/U$, since $x_n$ has odd grading and $H_*(C/U)/V\text{-torsion}$ is supported in $U$-grading zero. It follows that there exists some $y \in C$ and $m > 0$ for which $\d_V y = V^m f(x_n)$. Now define
\[
g(x_i) = 
\begin{cases}
	f(x_i) & \text {if } i=1, \dots, n \\
	y & \text {if } i=n+1.
\end{cases}
\]
Note that $b_{n+1} = m$. By construction, $g$ is a short local map.
\end{proof}

\begin{lemma}\label{lem:shortseqstandard}
Let $\{t_i\}_{i \in \N}$ be a sequence of integers with $t_i \rightarrow \infty$, and let
\[ f_i \co C(b_1, \dots, b_{n-1}, -t_i) \leadsto C \]
be a sequence of short local maps from standard complexes to a knot-like complex $C$. Then there exists a short local map
\[ f \co C(b_1, \dots, b_{n-1}, b_n) \leadsto C \]
for some $b_n > 0$.
\end{lemma}

\begin{proof}
As $i$ increases, the $V$-grading of the final generator $x_n$ of $C(b_1, \dots, b_{n-1}, -t_i)$ also increases. Since $C$ is finitely generated, it follows that for sufficiently large $i$, we have $f_i(x_n) = 0$. Restriction to the first $n-1$ generators thus yields a local map from the semistandard complex $C'(b_1, \dots, b_{n-1})$ to $C$. Now apply Lemma \ref{lem:semitostandard} to obtain the desired result.
\end{proof}

\begin{lemma}\label{shortseqsemistandard}
Let $\{t_i\}_{i \in \N}$ be a sequence of integers with $t_i \rightarrow \infty$, and let
\[ f_i \co C'(b_1, \dots, b_{n-1}, -t_i) \leadsto C \]
be a sequence of short local maps from semistandard complexes to a knot-like complex $C$. Then there exists a short local map
\[ f \co C(b_1, \dots, b_{n-1}) \leadsto C \]
from the standard complex $C(b_1, \dots, b_{n-1})$ to $C$.
\end{lemma}

\begin{proof}
As $i$ increases, the $U$-grading of the final generator $x_n$ of $C'(b_1, \dots, b_{n-1}, -t_i)$ also increases. Since $C$ is finitely generated, it follows that for sufficiently large $i$, we have $f_i(x_n) = 0$. Restriction to the first $n-1$ generators then yields a local map from the standard complex $C(b_1, \dots, b_{n-1})$ to $C$.
\end{proof}
\noindent
We are now ready to prove Proposition \ref{prop:sup}:

\begin{proof}[Proof of Proposition \ref{prop:sup}]\
We prove that the supremum in the definition of $a_i$ is always realized (modulo trailing zeros). We proceed by induction. Suppose that $a_1, \ldots, a_k$ are defined and nonzero. Let $\mathcal{F}$ be the family of standard complexes appearing in the definition of $a_{k+1}$. By examining the order on $\Zbang$, we see that the only subsets of $\Zbang$ which fail to attain their supremum are those which are unbounded below (in the usual sense). Hence the only case we have to worry about is when the family of $(k+1)$st symbols appearing in $\mathcal{F}$ has $\sup\bang$ equal to zero. 

If $k$ is odd, then truncating each element of $\mathcal{F}$ to its first $k + 1$ generators provides a family of standard complexes and local maps as in the statement of Lemma~\ref{lem:shortseqstandard}. This is a contradiction, since Lemma~\ref{lem:shortseqstandard} (combined with the extension lemma) then implies that the relevant $\sup\bang$ is strictly greater than zero. Thus, we may assume that $k$ is even. Then truncating each element of $\mathcal{F}$ to its first $k + 1$ generators yields a family of semistandard complexes to which we may apply Lemma~\ref{shortseqsemistandard}. In this situation, we see that $a_{k+1}$ is realized as a trailing zero, completing the proof.
\end{proof}
\section{Characterization of knot-like complexes up to local equivalence}\label{sec:char}

We now prove that every knot-like complex is locally equivalent to a standard complex. In fact, we prove a slightly stronger statement in Corollary~\ref{cor:splitting} below:

\begin{theorem}\label{thm:char}
Every knot-like complex is locally equivalent to a standard complex.
\end{theorem}

\begin{corollary}\label{cor:splitting}
Let $C$ be a knot-like complex, and assume $C$ is locally equivalent to $C(a_1, \dots, a_n)$. Then $C$ is homotopy equivalent to $C(a_1, \dots, a_n) \oplus A$, for some $\cR$-complex $A$.
\end{corollary}
\noindent
Theorem \ref{thm:char} immediately implies Theorem \ref{thm:localequivlex}:

\begin{proof}[Proof of Theorem \ref{thm:localequivlex}]
Following Section \ref{background}, to every knot in $S^3$, we can associate a knot-like complex. By Theorem \ref{thm:char}, every knot-like complex is locally equivalent to a standard complex, and by Proposition \ref{prop:lex}, standard complexes are ordered lexicographically. This proves Theorem~\ref{thm:localequivlex} modulo the claim that the standard complex associated to any knot is symmetric. We delay this until the end of the section; see Lemma~\ref{lem:symmetric}.
\end{proof}

Roughly speaking, we will show that if $C$ is a knot-like complex, then the numerical invariants $a_i(C)$ defined in the previous section compute successive parameters in the desired standard complex representative of $C$. Our main technical result will be to show that the $a_i$ (as defined previously) eventually become equal to zero:

\begin{proposition}\label{prop:aizero}
Let $C$ be a knot-like complex. Then $a_i(C) = 0$ for all $i$ sufficiently large.
\end{proposition}
\noindent
The proof of Proposition \ref{prop:aizero} will be given at the end of the section. First, we show how this implies Theorem \ref{thm:char}:

\begin{proof}[Proof of Theorem \ref{thm:char}]
Let $C$ be a knot-like complex with numerical invariants $a_i$. By Propositions \ref{prop:sup} and \ref{prop:aizero}, there exists some standard complex $C_1 \leq C$ which realizes the $a_i$. It is easily checked from the fact that standard complexes are lexicographically ordered that $C_1$ must be the maximal standard complex $\leq C$. Dualizing, let $C_2$ be the minimal standard complex with $C \leq C_2$. If $C_1 \neq C_2$, then (using the fact that standard complexes are lexicographically ordered) there exists a standard complex $C_3$ lying strictly between them. This complex contradicts either the maximality of $C_1$ or the minimality of $C_2$. Thus we must have the local equivalence $C_1 = C = C_2$.
\end{proof}

To prove the more refined Corollary \ref{cor:splitting}, we use the following series of lemmas concerning self-maps of standard complexes.

\begin{lemma}\label{lem:cutsequence}
Let 
\[ f \co C(b_1, \dots, b_n) \to C(b_1, \dots, b_n) \]
be a local map such that $f(x_i)$ is supported by $x_j$ for some $i \neq j$. Then
\[ (b_{i+1}, \dots, b_n) \lebang (b_{j+1}, \dots, b_n). \]
Here, we mean that $(b_{k+1}, \dots, b_n) = (0)$ if $k = n$.
\end{lemma}

\begin{proof}
First assume that $i$ is even. By grading considerations, this implies that $j$ is also even. We have the following casework:

\begin{enumerate}
\item
Suppose that $b_{i+1} < 0$. Then $\d_U x_i = U^{|b_{i+1}|} x_{i+1}$. Hence $\d_U f (x_i) = f \d_U x_i \in \im U^{|b_{i+1}|}$. Since $f(x_i)$ is supported by $x_j$, it follows that $\d_U x_j \in \im U^{|b_{i+1}|}$. This implies that $b_{j+1} \geqbang b_{i+1}$. (Here, we use the fact that no $U$-power of $x_{j+1}$ appears in $\d_U$ of any standard basis element other than $x_j$.) 
\item
Suppose that $b_{i+1} > 0$. Then $\d_U x_{i+1} = U^{b_{i+1}} x_i$. Hence $\d_U f (x_{i+1}) = U^{b_{i+1}} f(x_i)$ is supported by $U^{b_{i+1}} x_j$. In particular, $U^{b_{i+1}} x_j$ is in the image of $\d_U$, which implies that $b_{j+1} \geqbang b_{i+1}$. (Here, we use the fact that $x_{j+1}$ is the unique basis element whose image under $\d_U$ can be supported by a $U$-power of $x_j$.) 
\item
Suppose that $i = n$, so that $b_{i+1} = 0$. Then $\d_U x_i = 0$. Hence $\d_U f(x_i) = 0$. Since $f(x_i)$ is supported by $x_j$, it follows that $\d_U x_j = 0$. (Here, we  use the fact that no $U$-power of $x_{j+1}$ can appear in $\d_U$ of any standard basis element other than $x_j$.) This implies that $b_{j+1} > 0$.
\end{enumerate}
If strict inequality holds in any of the above cases, then we are done. On the other hand, if $b_{i+1} = b_{j+1}$, then it is easily seen that $f(x_{i+1})$ is supported by $x_{j+1}$, and we proceed inductively. By the hypothesis that $i \neq j$, the sequences $(b_{i+1}, \dots, b_n)$ and $(b_{j+1}, \dots, b_n)$ are of different lengths, and hence cannot be equal. The case $i$ odd is similar, with the role of $U$ played by $V$.
\end{proof}

\begin{lemma}\label{lem:selflocalmapinjective}
Any local map 
\[ f \co C(b_1, \dots, b_n) \to C(b_1, \dots, b_n) \]
must be injective.
\end{lemma}

\begin{proof}
Suppose not. Then there exists some linear combination $\sum_i r_ix_i$ with $r_i \in \cR$ such that $f(\sum_i r_ix_i)=0$. Since $f$ is graded, we may assume that $\sum_i r_ix_i$ is grading-homogenous, so that each $r_i$ is a monomial (that is, $r_i \in \{0, 1, U, U^2, \dots, V, V^2, \dots\}$). 

We impose a partial order on the set of monomials in $\cR$ by defining $1 > U > U^2 > \dots > 0$ and  $1 > V > V^2 > \dots > 0$. Among the nonzero coefficients $r_i$, choose a maximal element $r_{i_0}$ with respect to this partial order. Let $I = \{ j \mid r_j = r_{i_0} \}$. For each $j \in I$, consider $(b_{j+1}, \dots, b_n)$. Label the elements of $I = \{ j_1, \dots, j_m \}$ such that
\[ (b_{j_1+1}, \dots, b_n) \lebang (b_{j_2+1}, \dots, b_n)  \lebang \cdots \lebang (b_{j_m+1}, \dots, b_n) .\]
\noindent
Consider $f(x_{j_1})$. By Lemma \ref{lem:localsupport}, $f(x_{j_1})$ is supported by $x_{j_1}$. By Lemma \ref{lem:cutsequence}, $f(x_{j_i})$ for $i = 2, \dots, m$ cannot be supported by $x_{j_1}$. By the $\cR$-equivariance of $f$ and maximality of $r_{i_0}$, there is no other term in $f(\sum_{i \neq j_1} r_i x_i)$ that can cancel $r_{j_1} x_{j_1}$, contradicting the fact that $f(\sum_i r_ix_i)=0$. Hence $f$ must be injective.
\end{proof}

\noindent
We thus have:
\begin{lemma}\label{lem:selflocalmapisomorphism}
Any local self-map of a standard complex to itself is an isomorphism.
\end{lemma}
\begin{proof}
Let $f$ be a local self-map of a standard complex $C$. It is clear that $f$ must be absolutely $V$-graded. Hence $f$ restricted to each bigrading is a linear map from a finite-dimensional $\F$-vector space to itself, which is injective by Lemma~\ref{lem:selflocalmapinjective}. (Note that $C$ is finitely generated.) It follows that $f$ is surjective.
\end{proof}

\noindent
Using Lemma~\ref{lem:selflocalmapisomorphism}, we now prove Corollary \ref{cor:splitting}:

\begin{proof}[Proof of Corollary \ref{cor:splitting}]
By Theorem \ref{thm:char}, for a knot-like complex $C$, we have local maps
\[ f \co C(a_1, \dots, a_n) \to C \quad  \text{ and } \quad g \co C \to C(a_1, \dots, a_n). \]
Then $g \circ f$ is a local map from $C(a_1, \dots, a_n)$ to itself, which is an isomorphism by Lemma \ref{lem:selflocalmapisomorphism}. It follows that the short exact sequence
\[ 0 \to C(a_1, \dots, a_n) \xrightarrow{f} C \to C / \Im f \to 0 \]
splits.
\end{proof}

We now turn to the proof of Proposition \ref{prop:aizero}. We begin with the following lemma.

\begin{lemma}\label{lem:ImUVnotmax}
Let $C$ be a knot-like complex and let $a_i = a_i(C)$. Suppose we have a short local map
\[ f \co C(a_1, \dots, a_n) \leadsto C. \]
Then $f(x_i)$ is not in $\Im (U, V)$ for any $0 \leq i \leq n$.\footnote{Note that $0$ is considered to be in $\Im (U, V)$.} In particular, $f(x_i) \neq 0$ for $0 \leq i \leq n$. 
\end{lemma}

\begin{proof}
We first show that $f(x_i) \notin \Im U$. We proceed by contradiction. Let $j = \min \{ i \mid f(x_i) \in \Im U \}$ be the minimal index for which $f(x_j) \in \Im U$, and let $f(x_j) = U \eta_j$. (Note that $\eta_j$ is allowed to be zero.) Since $f$ is local, we have that $f(x_0) \neq 0 \in H_*(C/U)$, so $j \neq 0$.

Suppose that $j$ is odd. If $a_j \neq 1$, define a local map 
\[ g \co C'(a_1, \dots, a_j -1) \leadsto C \]
by setting $g(x_i) = f(x_i)$ for $1\leq i < j$ and $g(x_j) = \eta_j$. By the extension lemma, $g$ extends to a local map. This contradicts the maximality of $a_j$, since $a_j-1 \gebang a_j$. If $a_j = 1$, we have that $\d_U x_j = U x_{j-1}$. Since $\d_U f(x_j) =f \d_U (x_j) = U f(x_{j-1})$, we have $\d_U \eta_j = f(x_{j-1})$. Since $C$ is reduced, it follows that $f(x_{j-1}) \in \Im U$, contradicting the minimality of $j$.

Now suppose that $j$ is even. Assume $a_j < 0$. Since $f(x_j) \equiv 0 \mod U$, it is easily checked that the restriction of $f$ gives a local map
\[ g \co C'(a_1, \dots, a_{j-1}) \to C. \]
Applying Lemma~\ref{lem:semitostandard} and then the extension lemma shows that this contradicts the maximality of $a_j$. Thus, we may assume $a_j > 0$. Then 
\[ V^{a_j} f(x_{j-1}) = \d_V f(x_j) = \d_V U \eta_j = 0. \]
This implies that $f(x_{j-1}) \in \Im U$, contradicting the minimality of $j$. 

The case $f(x_i) \notin \Im V$ is similar. Indeed, let $j = \min \{ i \mid f(x_i) \in \Im V \}$, and let $f(x_j) = V \eta_j$. (Note that $\eta_j$ is allowed to be zero.) Since $H_*(C/U)$ does not have any $V$-nontorsion classes of positive grading, it follows that $j \neq 0$. The remainder of the proof follows by interchanging the roles of $U$ and $V$ in the argument above.
\end{proof}

\noindent
Before proceeding, we will need the following technical result which will allow us to rule out when certain complexes are $n$-maximal. The reader may wish to postpone reading the proof of Lemma~\ref{lem:blebangc} until after seeing its utilization in the proof of Proposition~\ref{prop:aizero}.

\begin{lemma}\label{lem:blebangc}
Let
\[ f \co C(b_1, \dots, b_m) \leadsto C \quad \text{ and } \quad g \co C(c_1, \dots, c_n) \leadsto C \]
be short local maps from standard complexes to a knot-like complex $C$. Let $\{y_i\}_{i=1}^m$ and $\{x_i\}_{i=1}^n$ denote the standard bases for $C(b_1, \dots, b_m)$ and $C(c_1, \dots, c_n)$, respectively. Suppose that $f(y_m) = g(x_n)$, and we have the inequality of reversed sequences
\[
(b_m, \dots, b_1) \lebang (c_n, \dots, c_1)
\]
with respect to the lexicographic order on $\Zbang$-valued sequences. Then $C(c_1, \dots, c_n)$ is not $n$-maximal (with respect to $C$).
\end{lemma}

\begin{proof}
Assume that the sequences $(b_m, \dots, b_1)$ and $(c_n, \dots, c_1)$ first differ in their $(\ell + 1)^\text{st}$ terms, so that $b_{m-i} = c_{n-i}$ for $0 \leq i < \ell$ and $b_{m-\ell} \lebang c_{n-\ell}$.\footnote{Here, $\ell \leq \min(m, n)$. Note that we allow $l = \min(m, n)$, with the convention that $b_0 = c_0 = 0$.} This means that the final $\ell + 1$ generators of $C(b_1, \dots, b_m)$ (and the arrows going between them) are isomorphic to the final $\ell + 1$ generators of $C(c_1, \dots, c_n)$. Our goal will be to define a new local map 
\[ h \co C(c_1, \dots, c_n) \to C \]
which has the property that $h(x_{n-i}) = g(x_{n-i}) + f(y_{m-i})$ for all $0 \leq i \leq \ell$. Since $f$ and $g$ are chain maps, it is evident that $h$ is a chain map, at least when restricted to the generators $x_{n-i}$ for $0 \leq i < \ell$. Below, we give the full verification and construction of $h$. In order to conclude the proof, we then note that $h(x_n) = g(x_n) + f(y_m) = 0$, and apply Lemma~\ref{lem:ImUVnotmax}.

We define $h$ on all generators except $x_{n - \ell - 1}$ as follows. Let
\begin{align}
\label{h1}&h(x_i) = g(x_i) \hspace{0.05cm} \ \ \qquad \qquad \qquad \qquad \text{for } 0 \leq i \leq n-\ell -2 \\
\label{h2}&h(x_{n-i}) = g(x_{n-i}) + f(y_{m-i}) \qquad \text{for }0 \leq i \leq \ell.
\end{align}
It is clear that the chain map condition $\d h = h \d$ holds for all generators $x_i$ with $i < n - \ell - 2$, as well as all generators with $i > n - \ell$. The main subtlety will thus be to define $h(x_{n - \ell - 1})$. We have the following casework:
\begin{equation}\label{h3}
h(x_{n-\ell-1}) = \begin{cases}
	g(x_{n-\ell-1}) + U^{b_{m-\ell}-c_{n-\ell}}f(y_{m-\ell-1}) &\text{if } c_{n-\ell}, b_{m-\ell} \text{ same sign} \\
	&\text{and }n-\ell \text{ odd}, \\
	g(x_{n-\ell-1}) + V^{b_{m-\ell}-c_{n-\ell}}f(y_{m-\ell-1}) &\text{if } c_{n-\ell}, b_{m-\ell} \text{ same sign} \\
	&\text{and }n-\ell \text{ even}, \\
	g(x_{n-\ell-1})  &\text{if } c_{n-\ell}, b_{m-\ell} \text{ different} \\
	&\text{signs}.
\end{cases}
\end{equation}
Here, we consider $b_0 = c_0 = 0$ to be of a different sign than either positive or negative. For the sake of concreteness, we explicitly describe $h$ in the two cases when $m < n$ and $n < m$. If $m < n$, then all three of (\ref{h1}), (\ref{h2}), and (\ref{h3}) are utilized when defining $h$. In particular, since $m$ and $n$ are both even and $\ell \leq \min(m, n)$, we have $n - \ell - 2 \geq 0$, and thus $h(x_0) = g(x_0)$. However, if $n < m$, then the form of $h$ may change slightly depending on the value of $\ell$. More precisely, if we are in the boundary case when $\ell = n$, then $h$ is defined on all generators by (\ref{h2}):
\begin{align*}
h(x_{n-i}) = g(x_{n-i}) + f(y_{m-i}) \qquad \qquad \text{for } 0 \leq i \leq n.
\end{align*}
Similarly, if $\ell = n - 1$, then only (\ref{h2}) and (\ref{h3}) are used:
\begin{align*}
&h(x_{n-i}) = g(x_{n-i}) + f(y_{m-i}) \qquad \qquad \text{for } 0 \leq i < n \\
&h(x_0) = g(x_0) + U^{b_{m-n+1}-c_{1}}f(y_{m-n}).
\end{align*}
Note that in all other cases, we again have $h(x_0) = g(x_0)$.

We now check that $h$ is a chain map. As in Section~\ref{sec:standards}, this consists of a number of technical but straightforward verifications. For simplicity, assume for the moment that $\ell < n - 1$. First consider the case when $n-\ell$ is odd. Note that this also implies $m - \ell$ is odd, so $m - \ell > 0$. It is clear that $\d_U h (x_{n-\ell-2}) = h \d_U (x_{n-\ell-2})$ and $\d_V h (x_{n-\ell}) = h \d_V (x_{n-\ell})$. For the remaining chain map conditions, we proceed with casework based on the signs of $c_{n-\ell-1}$ and $c_{n-\ell}$. First, we consider the possible signs of $c_{n-\ell-1}$ to verify that $h \d_V (x_{n-\ell-2}) = \d_V h (x_{n-\ell-2})$ and $h \d_V (x_{n-\ell-1}) = \d_V h (x_{n-\ell-1})$. We then consider the possible signs of $c_{n-\ell}$ to verify that $h \d_U (x_{n-\ell-1}) = \d_U h (x_{n-\ell-1})$ and $h \d_U (x_{n-\ell}) = \d_U h (x_{n-\ell})$.
\begin{enumerate}

\item Suppose $c_{n-\ell-1} < 0$. Then $\d_V x_{n-\ell-2} = V^{|c_{n-\ell-1}|} x_{n-\ell-1}$ and $\d_V x_{n-\ell-1} = 0$. Assume that $c_{n-\ell}$ and $b_{n-\ell}$ have the same sign. We compute
\begin{align*}
	h \d_V (x_{n-\ell-2}) &= V^{|c_{n-\ell-1}|} h(x_{n-\ell-1}) \phantom{\Big)} \\
		&= V^{|c_{n-\ell-1}|} \Big(g(x_{n-\ell-1}) + U^{b_{m-\ell}-c_{n-\ell}}f(y_{m-\ell-1})\Big) \\
		&= V^{|c_{n-\ell-1}|} g(x_{n-\ell-1}) \phantom{\Big)}\\
	\d_V h (x_{n-\ell-2}) &= \d_V g(x_{n-\ell-2}) \phantom{\Big)}\\
		&= g \d_V(x_{n-\ell-2}) \phantom{\Big)}\\
		&= V^{|c_{n-\ell-1}|} g( x_{n-\ell-1}). \phantom{\Big)}
\end{align*}
Similarly,
\begin{align*}
	h \d_V (x_{n-\ell-1}) &= 0 \phantom{\Big)} \\
	\d_V h (x_{n-\ell-1}) &= \d_V \Big(g(x_{n-\ell-1}) + U^{b_{m-\ell}-c_{n-\ell}}f(y_{m-\ell-1}) \Big) \\
		&= g\d_V(x_{n-\ell-1})  \phantom{\Big)} \\
		&= 0, \phantom{\Big)}
\end{align*}
as desired. If $c_{n-\ell}$ and $b_{m-\ell}$ have different signs, then the same computation holds, except that the $U^{b_{m-\ell}-c_{n-\ell}}f(y_{m-\ell-1})$ terms vanish.

\item Suppose $c_{n-\ell-1} > 0$. Then $\d_V x_{n-\ell-2} =0$ and $\d_V x_{n-\ell-1} = V^{c_{n-\ell-1}} x_{n-\ell-2}$. Assume that $c_{n-\ell}$ and $b_{n-\ell}$ have the same sign. We compute
\begin{align*}
	h \d_V (x_{n-\ell-2}) &= 0 \\
	\d_V h (x_{n-\ell-2}) &= \d_V g (x_{n-\ell-2}) = g \d_V (x_{n-\ell-2}) = 0.
\end{align*}
Similarly,
\begin{align*}
	h \d_V (x_{n-\ell-1}) &=  h(V^{c_{n-\ell-1}} x_{n-\ell-2}) = V^{c_{n-\ell-1}} g(x_{n-\ell-2}) \phantom{\Big)} \\
	\d_V h (x_{n-\ell-1}) &= \d_V \Big( g(x_{n-\ell-1}) + U^{b_{m-\ell}-c_{n-\ell}}f(y_{m-\ell-1}) \Big) \\
		&= g\d_V (x_{n-\ell-1}) \phantom{\Big)} \\
		&= V^{c_{n-\ell-1}} g(x_{n-\ell-2}), \phantom{\Big)}
\end{align*}
as desired. If $c_{n-\ell}$ and $b_{m-\ell}$ have different signs, then the same computation holds, except that the $U^{b_{m-\ell}-c_{n-\ell}}f(y_{m-\ell-1})$ terms vanish.

\item Suppose $c_{n-\ell} < 0$. Then $\d_U x_{n-\ell-1} = U^{|c_{n-\ell}|} x_{n-\ell}$ and $\d_U x_{n-\ell} = 0$. We compute
\begin{align*}
	h \d_U (x_{n-\ell-1}) &= h(U^{|c_{n-\ell}|} x_{n-\ell}) \\
		&= U^{|c_{n-\ell}|} \Big( g(x_{n-\ell}) + f(y_{m-\ell}) \Big) \\
	\d_U h (x_{n-\ell-1}) &= \d_U \Big( g(x_{n-\ell-1}) + U^{b_{m-\ell}-c_{n-\ell}}f(y_{m-\ell-1}) \Big) \\
		&= g\d_U(x_{n-\ell-1}) + U^{b_{m-\ell}-c_{n-\ell}}f \d_U (y_{m-\ell-1}) \phantom{\Big)} \\
		&=  U^{|c_{n-\ell}|}  g(x_{n-\ell}) + U^{b_{m-\ell}-c_{n-\ell}} f(U^{|b_{m-\ell}|} y_{m-\ell}) \phantom{\Big)} \\
		&= U^{|c_{n-\ell}|} \Big( g(x_{n-\ell}) + f(y_{m-\ell}) \Big).
\end{align*}
In the penultimate equality above, we are using the fact that $b_{m-\ell} \lebang c_{n-\ell} < 0$ to conclude that $\d_U (y_{m-\ell-1})=U^{|b_{m-\ell}|} y_{m-\ell}$; we use this again in the final equality to write $|b_{m-\ell}| = -b_{m-\ell}$ and $|c_{n-\ell}| = -c_{n-\ell}$. Similarly,
\begin{align*}
	h \d_U (x_{n-\ell}) &= 0 \phantom{\Big)} \\
	\d_U h (x_{n-\ell}) &= \d_U \Big( g(x_{n-\ell}) + f(y_{m-\ell}) \Big) = 0,
\end{align*}
where in the second equality above, we again use $b_{m-\ell} \lebang c_{n-\ell} < 0$.

\item Suppose $c_{n-\ell} > 0$. Then $\d_U x_{n-\ell-1} = 0$ and $\d_U x_{n-\ell} = U^{c_{n-\ell}} x_{n-\ell-1}$. We consider two further subcases, based on whether $b_{m-\ell} < 0$ or $b_{m-\ell} > 0$.
	\begin{enumerate} 
	
\item Suppose $b_{m-\ell} < 0$, so that $\d_U y_{m-\ell-1} = U^{|b_{m-\ell}|} y_{m-\ell}$ and $\d_U y_{m-\ell} = 0$. Then 
\begin{align*}
	h \d_U (x_{n-\ell-1}) &= 0 \\
	\d_U h (x_{n-\ell-1}) &= \d_U  g(x_{n-\ell-1}) = 0.
\end{align*}
Similarly,
\begin{align*}
	\qquad \qquad h \d_U (x_{n-\ell}) &=  h (U^{c_{n-\ell}} x_{n-\ell-1} ) = U^{c_{n-\ell}}  g(x_{n-\ell-1}) \\
	\d_U h (x_{n-\ell}) &= \d_U \Big(  g(x_{n-\ell}) + f(y_{m-\ell}) \Big) = U^{c_{n-\ell}} g(x_{n-\ell-1}),
\end{align*}
as desired.	

\item Suppose $b_{m-\ell} > 0$, so that $\d_U y_{m-\ell-1} = 0$ and $\d_U y_{m-\ell} = U^{b_{m-\ell}} y_{m-\ell-1}$. Then 
\begin{align*}
	\qquad \qquad h \d_U (x_{n-\ell-1}) &= 0 \\
	\qquad \qquad \d_U h (x_{n-\ell-1}) &= \d_U \Big( g(x_{n-\ell-1}) + U^{b_{m-\ell}-c_{n-\ell}}f(y_{m-\ell-1})  \Big) = 0.
\end{align*}
Similarly,
\begin{align*}
	\qquad h \d_U (x_{n-\ell}) &=  h (U^{c_{n-\ell}} x_{n-\ell-1} ) \\
		&= U^{c_{n-\ell}}  \Big( g(x_{n-\ell-1}) + U^{b_{m-\ell}-c_{n-\ell}}f(y_{m-\ell-1}) \Big) \\
		&= U^{c_{n-\ell}}  g(x_{n-\ell-1}) + U^{b_{m-\ell}} f(y_{m-\ell-1}) \\
	\qquad \d_U h (x_{n-\ell}) &= \d_U \Big(  g(x_{n-\ell}) + f(y_{m-\ell}) \Big) \\
		&= U^{c_{n-\ell}} g(x_{n-\ell-1}) + U^{b_{m-\ell}} f  (y_{m-\ell-1}),
\end{align*}
as desired.
\end{enumerate}

\end{enumerate}
This shows that $h$ is a chain map, at least when $\ell < n - 1$ and $n-\ell$ is odd. The proof when $n-\ell$ is even follows by interchanging the roles of $U$ and $V$. (There is a slight re-interpretation of Case (4) when $\ell = m$, which we leave to the reader.)

Finally, we consider the remaining cases when $\ell = n$ or $\ell = n - 1$. If $\ell = n$, then the only nontrivial check is to show that $\d h (x_0) = h \d (x_0)$. In this situation, we have $c_{m-n} \lebang b_0 = 0$. First, suppose that $c_{n - m + 1} = b_1 < 0$. Then
\begin{align*}
	\d h(x_0) & = \d f(x_0) + \d g(y_{n-m}) = U^{|b_1|} f(x_1) + U^{|c_{n-m+1}|} g(y_{n-m+1}) \\
	h \d (x_0) &= h(U^{|b_1|} x_1) = U^{|b_1|} h(x_1) =  U^{|b_1|} \Big( f(x_1) + g(y_{n-m+1}) \Big).
\end{align*}
The case $b_1 > 0$ is analogous. The situation when $\ell = n -1$ is similar in flavor, and we leave it to the reader.

We now claim that $h$ is a local map. If $\ell < n - 1$, then $h(x_0) = g(x_0)$, and so clearly $h$ is local. If $\ell = n - 1$, then $h(x_0) = g(x_0) + U^{b_{m-n+1} - c_1}f(y_{m-n})$. Hence $h(x_0)$ and $g(x_0)$ are equal in $C/U$, and $h$ is again local. Finally, if $\ell = n$, then $h(x_0) = g(x_0) + f(y_{m - n})$. Since $b_{m - n} \lebang c_0 = 0$, we have that $\d_V y_{m - n} = 0$ and $\d_V y_{m - n - 1} = V^{|b_{m-n}|} y_{m-n}$. Hence $f(y_{m-n})$ is a $V$-torsion cycle in $H_*(C/U)$. Since $g(x_0)$ generates $H_*(C/U)/V\text{-torsion}$, this shows that $h$ is local, as desired. 

By construction, $h(x_n) = f(y_m) + g(x_n) = 0$. Applying Lemma \ref{lem:ImUVnotmax}, we conclude that $C(c_1, \dots, c_n)$ is not $n$-maximal with respect to $C$. 
\end{proof}
\noindent
We are now ready to prove Proposition \ref{prop:aizero}:

\begin{proof}[Proof of Proposition \ref{prop:aizero}]
We proceed by contradiction. Suppose that $a_i \neq 0$ for all indices $i$. Let $n$ be very large. By Proposition \ref{prop:sup}, we have a short local map
\[ g \co C(a_1, \dots, a_n) \leadsto C. \]
Since $C$ is finitely generated, it follows from Lemma~\ref{lem:ImUVnotmax} that for $n$ sufficiently large, we must have $g(x_m) = g(x_n)$ for some $m < n$. Indeed, Lemma \ref{lem:ImUVnotmax} implies that the gradings of the $g(x_i)$ must lie in a bounded interval, since otherwise some $g(x_i)$ would be in $\Im U$ or $\Im V$. Hence $g(x_m) = g(x_n)$ for some $m < n$.


Consider the short local map
\[ f \co C(a_1, \dots, a_m) \leadsto C \]
obtained by restricting $g$. On one hand, $C(a_1, \dots, a_m)$ and $C(a_1, \dots, a_n)$ are evidently $m$- and $n$-maximal with respect to $C$. However, since $m \neq n$, we have that either $(a_m, \dots, a_1) \lebang (a_n, \dots, a_1)$ or  $(a_n, \dots, a_1) \lebang (a_m, \dots, a_1)$. Hence we may apply Lemma~\ref{lem:blebangc}, either with the maps $f$ and $g$, or vice-versa. This gives a contradiction.
\end{proof}

We now justify Remark~\ref{remark:absrel} and show that if $C_1$ and $C_2$ are locally equivalent via maps $f$ and $g$, then $f$ and $g$ take $U$-tower classes to $U$-tower classes:

\begin{lemma}\label{lem:loceqsymmetric}
Let $C_1$ and $C_2$ be knot-like complexes. Suppose $C_1$ and $C_2$ are locally equivalent via $f$ and $g$. Then $f$ and $g$ induce isomorphisms on $H_*(C_i/V)/U\text{-torsion}$.
\end{lemma}
\begin{proof}
By passing to the same local representative, we may assume that $C_1$ is a standard complex. Then $g \circ f$ is a local map from a standard complex to itself, which is an isomorphism by Lemma~\ref{lem:selflocalmapisomorphism}. In particular, $g \circ f$ induces an isomorphism from $H_*(C_1/V)/U\text{-torsion}$ to itself, factoring through the composition
\[
H_*(C_1/V)/U\text{-torsion} \xrightarrow{f} H_*(C_2/V)/U\text{-torsion} \xrightarrow{g} H_*(C_1/V)/U\text{-torsion}.
\]
Since each of the above terms consists of a single $U$-tower, it is clear that the induced maps must individually be isomorphisms.
\end{proof}

\noindent
Finally, we show that the standard complex associated to any knot is symmetric:

\begin{lemma}\label{lem:symmetric}
Let $K$ be a knot in $S^3$, and let $C = C(a_1, \ldots, a_n)$ be the standard complex representative of $\CFKUV(K)$. Then $C$ is symmetric.
\end{lemma}
\begin{proof}
Given Lemma~\ref{lem:loceqsymmetric}, it is clear that the definition of local equivalence is in fact completely symmetric with respect to interchanging the roles of $U$ and $V$. That is, we may require the maps $f$ and $g$ in Definition~\ref{def:le} to be absolutely $U$- and $V$-graded, and induce isomorphisms on both $H_*(C_i/U)/V\text{-torsion}$ and $H_*(C_i/V)/U\text{-torsion}$. Now suppose that $f$ and $g$ are such local equivalences between $C$ and $\CFKUV(K)$. Then it is not hard to see that we have local equivalences between these two complexes with the roles of $U$ and $V$ reversed; i.e., 
\[
\overline{C} \sim \overline{\CFKUV}(K).
\]
However, we already know that $\CFKUV(K)$ is homotopy equivalent to $\overline{\CFKUV}(K)$, so $C \sim \overline{C}$. It is easily checked that passing from $C$ to $\overline{C}$ reverses the order of the standard complex parameters, showing that $C$ is symmetric, as desired.
\end{proof}
\section{Homomorphisms}\label{sec:homs}
In this section, we construct an infinite family of linearly independent homomorphisms from $\KL$ to $\Z$.

\subsection{Some $\Z$-valued homomorphisms}

We begin with the following definition.

\begin{definition}\label{def:homs}
Let $C = C(a_1, \dots, a_n)$ be a standard complex. Define
\[ \varphi_j(C) = \# \{ a_i \mid a_i = j, i \text{ odd} \} -  \# \{ a_i \mid a_i = - j, i \text{ odd} \}. \]
That is, $\varphi_j(C)$ is the signed count of the number of times that $j$ appears as an odd parameter $a_{2k+1}$. Equivalently, $\varphi_j(C)$ is the signed count of horizontal arrows of length $j$. If $C$ is any knot-like complex, then we define $\varphi_j(C)$ by passing to the standard complex representative of $C$ afforded by Theorem~\ref{thm:char}.
\end{definition}
\noindent
The goal of this section is to prove the following theorem.

\begin{theorem}\label{thm:homs}
For each $j \in \N$, the function
\[ \varphi_j \co \KL \to \Z \]
is a homomorphism.
\end{theorem}

\noindent
Note that the product of two standard complexes is not a standard complex. Thus, to compute $\varphi_j(C_1 \otimes C_2)$ directly, we would first have to determine the standard complex representative of $C_1 \otimes C_2$. However, it turns out that we do not currently have an explicit description of the group law on $\KL$ in terms of the standard complex parameters (see Section~\ref{sec:furtherremarks}). Instead, we prove Theorem~\ref{thm:homs} by expressing each $\varphi_j$ as a linear combination of other auxiliary homomorphisms. The construction of these (and the proof that they are additive) will occupy our attention for the next two subsections.

Before proceeding, we show that Theorem~\ref{thm:main} follows readily from Theorem~\ref{thm:homs}:

\begin{proof}[Proof of Theorem \ref{thm:main}]
By Theorem \ref{thm:concle} and the behavior of $\CFKUV(K)$ under connected sum, we have a homomorphism
\[ \cC \to \KL \]
sending $[K]$ to $[\CFKUV(K)]$. Now compose with $\varphi_j$. (We henceforth abuse notation slightly and also refer to the composition $\cC \to \KL \to \Z$ as $\varphi_j$.) Surjectivity of
\[ \bigoplus_{j=1}^\infty  \varphi_j \co \cC \to \bigoplus_{j=1}^\infty  \Z \]
follows from the observation that $\varphi_j(T_{i+1, i+2}\#-T_{i, i+1}) =\delta_{ij}$ (see Example \ref{ex:torusknots}), or alternatively by considering the knots in Proposition \ref{prop:cablesofD}.
\end{proof}

We now introduce the first of our auxiliary homomorphisms:

\begin{definition}\label{def:P}
Let $C$ be a knot-like complex and let $C(a_1, \dots, a_n)$ be the standard complex representative of $C$ given by Theorem \ref{thm:char}. Define
\[ P(C) = -2\sum_{j > 0} j \varphi_j(C) + \sum_{i=1}^n \sgn a_{i}. \]
\end{definition}

\noindent
It is clear that $\varphi_j$ is an invariant of the local equivalence class of $C$. To see that $P$ is a homomorphism, we use the following alternative definition.

\begin{lemma}\label{lem:Pgrading}
The integer $P(C)$ is equal to the $U$-grading of a $U$-tower generator. 
\end{lemma}

\begin{proof}
By Corollary \ref{cor:splitting}, $C$ is homotopy equivalent to $C(a_1, \dots, a_n) \oplus A$, where $a_i = a_i(C)$ and $A$ is some $\cR$-complex. Since $C$ is a knot-like complex, $U^{-1}H_*(C) \cong \F[U, U^{-1}]$, and so the $U$-nontorsion classes in $C$ are supported by the standard summand $C(a_1, \dots, a_n)$. It is then clear that $x_n$ is a $U$-tower generator in $C$. A straightforward computation shows that $\gr_U(x_n)$ is given by the expression in Definition~\ref{def:P}.
\end{proof}

\noindent
Given this, we immediately have:

\begin{proposition}
The function $P \co \KL \to 2\Z$ is a surjective homomorphism.
\end{proposition}

\begin{proof}
The fact that $P$ is a homomorphism follows from the K\"unneth formula. To see that $P$ is surjective, we observe that $P(C(1, -1)) = -2$.
\end{proof}

\noindent
Before proceeding, we show  Proposition \ref{prop:tau} from the introduction:
\proptau*

\begin{proof}[Proof of Proposition \ref{prop:tau}]
It is sufficient to consider the local equivalence class of $\CFKUV(K)$. Let $C=C(a_1, \dots, a_n)$ be the local equivalence class of $\CFKUV(K)$. Then $C$ is symmetric, so $\sum_{i=1}^n \sgn a_i =0$ and $P(C) = \gr_U(x_n) = -2\tau(K)$.
\end{proof}

\subsection{Shift homomorphisms} We now introduce an auxiliary family of endomorphisms $\shift_m \co \KL \rightarrow \KL$ for $m \in \N$. Composing these with $P$, we obtain an infinite sequence of homomorphisms $P \circ \shift_m: \KL \rightarrow 2\Z$. In the next subsection, we show that the $\varphi_j$ are certain linear combinations of the $P \circ \shift_m$ (divided by two). Our present goal will be to define the $\shift_m$ and show that they are additive. This will be the most technical part of the argument, and will require the introduction of several auxiliary definitions.
\begin{definition}\label{def:shift}
Let $C=C(a_1, \dots, a_n)$ be a standard complex. Let $\shift_m(C)$ be the standard complex given by
\[ \shift_m (C) =  C(a_1', \dots, a_n'), \]
where
\[ a_i' = 
\begin{cases}
	a_i +1 &\text{ if } a_i \geq m \\
	a_i -1 &\text{ if } a_i \leq -m \\
	a_i  &\text{ if } |a_i| < m \\  
\end{cases}
\]
\noindent
That is, $\shift_m$ fixes $U^n$- and $V^n$-arrows for $n < m$ and takes $U^n$- and $V^n$-arrows to $U^{n+1}$- and $V^{n+1}$-arrows respectively for $n \geq m$.
\end{definition}
\noindent
The majority of this subsection will be devoted to proving the following theorem.

\begin{theorem}\label{thm:shift}
For all $m \geq 1$, the function $\shift_m \co \KL \to \KL$ is a homomorphism, that is, for knot-like complexes $C_1$ and $C_2$, we have the local equivalence
\[ \shift_m(C_1 \otimes C_2) \sim \shift_m(C_1) \otimes \shift_m(C_2). \]
\end{theorem}
\noindent
It will also be helpful to decompose $\shift_m$ as a composition of a shift in $U$ and a shift in $V$ (denoted $\shift_{U,m}$ and $\shift_{V,m}$, respectively):

\begin{definition}\label{def:UVshift}
Given a standard complex $C=C(a_1, \dots, a_n)$, let
\[ \shift_{U,m}(C) = C(a_1', \dots, a_n') \]
where for $i$ odd,
\[ a_i' = 
\begin{cases}
	a_i +1 &\text{ if } a_i \geq m \\
	a_i -1 &\text{ if } a_i \leq -m \\
	a_i  &\text{ if } |a_i| < m \\  
\end{cases}
\]
and for $i$ even,
\[ a_i' = a_i. \]
Similarly,  let
\[ \shift_{V,m}(C) = C(a_1', \dots, a_n') \]
where for $i$ even,
\[ a_i' = 
\begin{cases}
	a_i +1 &\text{ if } a_i \geq m \\
	a_i -1 &\text{ if } a_i \leq -m \\
	a_i  &\text{ if } |a_i| < m \\  
\end{cases}
\]
and for $i$ odd,
\[ a_i' = a_i. \]
\end{definition}
\noindent
It follows from the definitions that $\shift_m = \shift_{V,m} \circ \shift_{U,m}$.

\begin{lemma}\label{lem:shiftdual}
Let $C = C(a_1, \dots, a_n)$ be a standard complex. Then 
\[ \shift_{U,m} (C)^\vee = \shift_{U,m} (C^\vee) \quad \text{ and } \quad \shift_{V,m} (C)^\vee = \shift_{V,m} (C^\vee).\]
\end{lemma}

\begin{proof}
The result follows from the definition of $\shift_{U,m}$ and $\shift_{V,m}$ combined with Lemma \ref{lem:dual}.
\end{proof}

We now introduce some convenient terminology:

\begin{definition}
Let $C$ be a knot-like complex (not necessarily a standard complex) with an $\cR$-basis $\{ x_i \}$. We say that $\{x_i\}$ is \emph{$U$-simplified} if for each $x_i$, exactly one of the following holds:
\begin{enumerate}
	\item \label{it:notker} $\d_U x_i = U^k x_j$ for some $j$ and $k$,
	\item \label{it:image} $\d_U x_j = U^k x_i$ for some $j$ and $k$, or
	\item \label{it:gen} $\d_U x_i = 0$ and $x_i \notin \Im \d_U$. 
\end{enumerate}
If $\d x_i = U^k x_j$ (or vice-versa), we say that $x_i$ and $x_j$ are \emph{$U$-paired}. Since $H_*(C/V)$ has a single $U$-tower, it follows that at most one of the $x_i$ satisfies \eqref{it:gen}. We define a \emph{$V$-simplified basis} and \emph{$V$-paired} basis elements analogously. (See for example the proof of Lemma~\ref{lem:totalorder}.)
\end{definition}

\begin{example}
Let $C=C(a_1, \dots, a_n)$ be a standard complex with preferred basis $\{x_i\}_{i=0}^n$; this basis is clearly both $U$- and $V$-simplified.  We will find it convenient to re-label our basis elements slightly. We denote the \emph{$U$-simplified basis} \smash{$\{ w, y_i, z_i \}_{i=1}^{n/2}$} for $C$ by
\[ w = x_n \] 
and for each $1 \leq i \leq n/2$, 
\begin{align*}
	y_i &= \begin{cases}
			x_{2i-1} \quad \text{ if } a_{2i-1} > 0 \\
			x_{2i-2}\quad \text{ if } a_{2i-1} < 0 
		\end{cases}\\
	z_i &= \begin{cases}
			x_{2i-1} \quad \text{ if } a_{2i-1} < 0 \\
			x_{2i-2} \quad \text{ if } a_{2i-1} > 0. 	
		\end{cases}
\end{align*}
Set-wise, the $U$-simplified basis is of course identical to the standard preferred basis, but we fix notation so that $\d_U y_i = U^{|a_{2i-1}|} z_i$. (That is, $y_i$ and $z_i$ are $U$-paired.) We can likewise define the \emph{$V$-simplified basis} in the obvious way.

\end{example}

\begin{definition}\label{def:sum}
For $C = C(a_1, \dots, a_n)$, let $\{ w, y_i, z_i \}$ and $\{ w', y'_i, z'_i \}$ be the $U$-simplified bases for $C$ and $\shift_{U,m}(C)$ respectively. Define an $\cR$-module map
\[ s_{U,m} \co C \rightarrow \shift_{U,m}(C) \]
by sending
\[ s_{U,m}(r) = r' \]
for each $r \in \{ w, y_i, z_i \}$, and extending $\cR$-linearly. That is, $s_{U, m}$ simply effects the correspondence between the unprimed generators of $C$ and the primed generators of $\shift_{U, m}(C)$. Note that $s_{U,m}$ induces an isomorphism of ungraded $\cR$-modules, although we stress that $s_{U,m}$ is \textit{not} graded (even relatively). Furthermore, it is easily checked that $s_{U,m} \d_V = \d_V s_{U,m}$. On the other hand, $s_{U,m}$ does not commute with $\d_U$. Explicitly, we have
\begin{align*}
&s_{U,m} (\d_U y_i) = s_{U, m} (U^{|a_{2i-1}|} z_i) = U^{|a_{2i-1}|} z_i' \\
&\d_U (s_{U,m} y_i )= \d_U (y_i') = U^{|a_{2i-1}'|} z_i'.
\end{align*}
Note that the above expressions may differ by a power of $U$, depending on the value of $|a_{2i-1}|$.
\end{definition}

\begin{example}\label{examp:Usimplifiedtensor}
Let $C_1 = C(a_1, \dots, a_{n_1})$ and $C_2 = C(b_1, \dots, b_{n_2})$ be standard complexes. Abusing notation slightly, let $w, y_i,$ and  $z_i$ denote the $U$-simplified bases for both $C_1$ and $C_2$; it will be clear from context which generators lie in $C_1$ and $C_2$. Then the obvious tensor product basis for $C_1 \otimes C_2$ is not $U$-simplified. Instead, we define a $U$-simplified basis for $C_1 \otimes C_2$ as follows. For $1 \leq i \leq n_2/2$, let
\begin{align*}
	\alpha_i &= w \otimes y_i\\
	\beta_i &= w \otimes z_i,
\end{align*}
and for $1 \leq i \leq n_1/2$, let
\begin{align*}
	\gamma_i &= y_i \otimes w \\
	\delta_i &= z_i \otimes w.
\end{align*}
For $1 \leq i \leq n_1/2$ and $1 \leq j \leq n_2/2$, define
\begin{align*}
	\epsilon_{i,j} &=y_i \otimes y_j \\
	\zeta_{i,j} &= \begin{cases}
				U^{|b_{2j-1}| - |a_{2i-1}|} y_i \otimes  z_j +  z_i \otimes y_j   \quad \text{ if } |a_{2i-1}| \leq |b_{2j-1}| \\
				y_i \otimes z_j + U^{ |a_{2i-1}| - |b_{2j-1}|} z_i \otimes y_j  \quad \text{ if } |a_{2i-1}| > |b_{2j-1}| 
			\end{cases}
\end{align*}
and
\begin{align*}
	\eta_{i,j} &= \begin{cases}
				y_i \otimes z_j \quad \text{ if } |a_{2i-1}| \leq |b_{2j-1}| \\
				z_i \otimes y_j \quad \text{ if } |a_{2i-1}| > |b_{2j-1}| 
			\end{cases}\\
	\theta_{i,j} &= z_i \otimes z_j.
\end{align*}
Finally, let
\[
	\omega = w \otimes w. 
\]
Note that the following basis elements are $U$-paired:
\[ \{ \alpha_i, \beta_i \},  \quad \{ \gamma_i, \delta_i \}, \quad \{ \epsilon_{i,j}, \zeta_{i,j} \}, \quad \{ \eta_{i,j}, \theta_{i,j} \}. \]
For notational convenience, we relabel the basis elements
\begin{align*}
	\{ \kappa_\ell \} &= \{ \alpha_i \} \cup \{ \gamma_i \} \cup \{ \epsilon_{i,j} \} \cup \{ \eta_{i,j} \} \\
	\{ \lambda_\ell \} &= \{ \beta_i \} \cup \{ \delta_i \} \cup \{ \zeta_{i,j} \} \cup \{ \theta_{i,j} \}
\end{align*}
so that $\{ \omega, \kappa_\ell, \lambda_\ell \}$ is a $U$-simplified basis and $\d_U \kappa_\ell = U^{e_\ell} \lambda_{\ell}$ for some $e_{\ell}$. The reader should check that if $\kappa_{\ell}$ is one of $\epsilon_{i,j}$ or $\eta_{i, j}$, then
\[
e_{\ell} = \min(|a_{2i-1}|, |b_{2j-1}|).
\]
If $\kappa_{\ell}$ is an $\alpha_i$, then $e_{\ell} = |b_{2i-1}|$, while if $\kappa_\ell$ is a $\gamma_i$, then $e_{\ell} = |a_{2i-1}|$.

We analogously define a $U$-simplified basis $\{ \alpha'_i, \beta'_i, \gamma'_i, \delta'_i, \epsilon'_{i,j}, \zeta'_{i,j}, \eta'_{i,j}, \theta'_{i,j}, \omega' \}$ for $\shift_{U,m}(C_1) \otimes \shift_{U,m}(C_2)$ by considering both factors as standard complexes in their own right. (That is, $\alpha'_i = w' \otimes y'_i$, and so on.) We re-label this basis $\{ \omega', \kappa'_{\ell}, \lambda'_\ell \}$ as before, so that $\kappa_\ell'$ and $\lambda_\ell'$ are $U$-paired. As above, we have $\d_U \kappa_\ell' = U^{e_\ell'} \lambda_{\ell}'$, where 
\[
e_\ell' = \min(|a_{2i-1}'|, |b_{2j-1}'|),
\]
whenever $\kappa_{\ell}'$ is one of $\epsilon_{i,j}'$ or $\eta_{i, j}'$ (similarly for the other cases). An examination of Definition~\ref{def:UVshift} then shows that we may write
\[
e_\ell' = e_\ell+\tau(\ell), 
\]
where
\[ \tau(\ell) = \begin{cases}
			0 &\text{ if } e_\ell < m \\
			1 &\text{ if } e_\ell \geq m.			
	\end{cases}
\]	
\end{example}

\begin{definition}
Let $C_1$ and $C_2$ be standard complexes. Define an $\cR$-module map
\[ \sigma_{U,m} \co C_1 \otimes C_2 \to \shift_{U,m}(C_1) \otimes \shift_{U,m}(C_2) \]
by sending
\[ \sigma_{U,m}( \xi ) = \xi ' \]
for $\xi \in \{ \omega, \kappa_\ell, \lambda_\ell \}$, and extending $\cR$-linearly. As in Definition~\ref{def:sum}, $\sigma_{U,m}$ induces an isomorphism of ungraded $\cR$-modules. Furthermore, we claim that $\sigma_{U,m} \d_V = \d_V \sigma_{U,m}$. To see this, observe that
\[
\sigma_{U, m} \equiv s_{U, m} \otimes s_{U, m} \bmod U.
\]
Indeed, this congruence is obviously an equality for all basis elements not of the form $\gamma_{i,j}$ or $\eta_{i, j}$. For $\eta_{i, j}$, we again have equality using the fact that $|a_{2i-1}| \leq |b_{2j-1}|$ if and only if $|a_{2i-1}'| \leq |b_{2j-1}'|$. For basis elements of the form $\gamma_{i,j}$, a straightforward casework check establishes the congruence. The fact that $s_{U,m}$ commutes with $\d_V$ then shows that $\sigma_{U,m} \d_V = \d_V \sigma_{U,m}$. Again, however, note that $\sigma_{U,m}$ does not commute with $\d_U$.
\end{definition}

We now introduce an auxiliary technical definition which we will need to prove Theorem~\ref{thm:shift}:

\begin{definition}\label{def:almostchainmap}
An \emph{almost chain map} $f \co C(a_1, \dots, a_n) \to C$ from a standard complex with preferred basis $\{x_i\}_{i=0}^n$ to a knot-like complex is an ungraded $\cR$-module map such that for $1 \leq i \leq n$:
\begin{enumerate}
	\item for $i$ odd,
		\begin{enumerate}
			\item if $a_i < 0$, that is, $\d_U x_{i-1} = U^{|a_i|}x_i$, we have 
		\[ \d_U f (x_{i-1}) \equiv U^{|a_i|} f(x_{i}) \mod U^{|a_i|+1}, \]
			\item if $a_i > 0$, that is, $\d_U x_{i} = U^{|a_i|}x_{i-1}$, we have 
			\[ \d_U f (x_{i}) \equiv U^{|a_i|} f(x_{i-1}) \mod U^{|a_i|+1}, \]
		\end{enumerate}
	\item for $i$ even, 
		\begin{enumerate}
			\item if $a_i < 0$, that is, $\d_V x_{i-1} = V^{|a_i|}x_i$, we have 
		\[ \d_V f (x_{i-1}) \equiv V^{|a_i|} f(x_{i}) \mod V^{|a_i|+1}, \]
			\item if $a_i > 0$, that is, $\d_V x_{i} = V^{|a_i|}x_{i-1}$, we have 
			\[ \d_V f (x_{i}) \equiv V^{|a_i|} f(x_{i-1}) \mod V^{|a_i|+1}. \]
		\end{enumerate}
\end{enumerate} 
We stress that an almost chain map is \textit{not} in general a chain map, and may not even be grading-homogeneous.
\end{definition}

The main import of the (admittedly unmotivated) notion of an almost chain map will be the following lemma, which explains how to extract a genuine chain map from a given almost chain map. In our context, it will be easier to construct almost chain maps, which is why we have introduced Definiton~\ref{def:almostchainmap}. In what follows, let $[x]_{p, q}$ denote the homogeneous part of $x$ in bigrading $(u, v)$.

\begin{lemma}\label{lem:almosttolocal}
Let $f \co C(a_1, \dots, a_n) \to C$ be an almost chain map. Let $(u_i, v_i)$ be the bigrading of the generator $x_i$ in $C(a_1, \dots, a_n)$. Suppose that $[f(x_0)]_{u_0, v_0}$ represents a $V$-tower class in $C$ and $\d_U[f(x_n)]_{u_n, v_n} = 0$. Then there exists a genuine local map 
\[ g \co C(a_1, \dots, a_n) \to C \]
such that $g(x_i) \equiv [f(x_i)]_{u_i, v_i} \bmod (U,V)$ for all $0 \leq i \leq n$.
\end{lemma}

\begin{proof}
For each $0 \leq i \leq n$, consider the ansatz:
\[
g(x_i) = [f(x_i)]_{u_i, v_i} + Up_i + Vq_i
\]
where $p_i$ and $q_i$ are undetermined elements of $C(a_1, \dots, a_n)$ with bigrading $(u_i, v_i)$. In order to determine $p_i$ and $q_i$, we substitute our ansatz into the chain map condition for $g$. We begin by using the condition $\d_U g = g\d_U$ to help determine the $p_i$:
\begin{enumerate}
\item Let $i$ be odd, and suppose $a_i < 0$. Then $\d_U x_{i-1} = U^{|a_i|} x_i$ and $\d_U x_i = 0$. Using Definition~\ref{def:almostchainmap}, write
\[
\d_U f (x_{i-1}) = U^{|a_i|} f(x_{i}) + U^{|a_i| + 1}\eta_i
\]
for some (possibly non-homogeneous) element $\eta_i \in C(a_1, \dots, a_n)$. Note that since $\d_U^2 = 0$, we have $\d_U f(x_i) + U\d_U \eta_i = 0$. We now compute:
\begin{align*}
g(\d_Ux_{i-1}) &= U^{|a_i|} g(x_i) \\
&= U^{|a_i|} ([f(x_i)]_{u_i, v_i}  + Up_i + Vq_i) \\
&= U^{|a_i|} [f(x_i)]_{u_i, v_i} + U^{|a_i| + 1}p_i \\
\d_U g(x_{i-1}) &= \d_U ([f(x_{i-1})]_{u_{i-1}, v_{i-1}} + Up_{i-1} + Vq_{i-1})\\
&= U^{|a_i|}[f(x_i)]_{u_i, v_i} + U^{|a_i| + 1} [\eta_i]_{u_i + 2, v_i + 2} + U\d_Up_{i-1},
\end{align*}
where in the last line, we have used the fact that $\d_U f (x_{i-1}) = U^{|a_i|} f(x_{i}) + U^{|a_i| + 1}\eta_i$. We likewise compute
\begin{align*}
g(\d_Ux_i) &= g(0) = 0 \\
\d_U g(x_i) &= \d_U([f(x_i)]_{u_i, v_i} + Up_i + Vq_i) \\
&= \d_U [f(x_i)]_{u_i, v_i} + U\d_U p_i.
\end{align*}
Examining the first pair of equalities above, we see that it suffices to set $p_{i-1} = 0$ and $p_i = [\eta_i]_{u_i + 2, v_i + 2}$. The second pair of equalities then follows from the fact that $\d_U f(x_i) + U\d_U \eta_i = 0$.
\item Let $i$ be odd, and suppose $a_i > 0$. Then $\d x_{i} = U^{a_i} x_{i-1}$ and $\d_U x_{i-1} = 0$. A similar analysis as above (interchanging the roles of $i$ and $i-1$ and replacing $|a_i|$ with $a_i$) shows that if we set $p_{i-1} = [\eta_{i-1}]_{u_{i-1}+2, v_{i-1}+2}$ and $p_i = 0$, then we have $(g\d_U + \d_Ug)(x_{i-1}) = (g \d_U + \d_U g)(x_i) = 0$.
\end{enumerate}
In this manner, by considering all odd indices $1 \leq i \leq n$, we see that we can choose the $p_i$ for $0 \leq i < n$ so that $(g \d_U + \d_Ug)(x_i) = 0$ for all $0 \leq i < n$. Define $p_n = 0$. Then $\d_U g(x_n) = \d_U [f(x_n)]_{u_i, v_i} = 0$ by hypothesis, while $g\d_U(x_n) = 0$. This establishes the $\d_U$-condition for all generators $x_i$.

Interchanging the roles of $U$ and $V$, an analogous argument (where we consider the case when $i$ is even) allows us to choose the $q_i$ such that $(g\d_V + \d_Vg)(x_i) = 0$ for all $0 \leq i \leq n$. (To establish the $\d_V$-condition for $x_0$, we use the fact that $\d_V [f(x_0)]_{u_0, v_0} = 0$, since $[f(x_0)]_{u_0, v_0}$ represents a $V$-tower class in $C$ by hypothesis.) By construction, $g$ is a graded, $\cR$-equivariant chain map which is clearly local. This completes the proof.
\end{proof}

Now let $C_3 = C(c_1, \dots, c_{n_3})$ be a standard complex, and let $f \co C_3 \to C_1 \otimes C_2$ be a local map. Our goal will be to construct a shifted map $f_{U, m}$ from $\shift_m C_3$ to $\shift_m C_1 \otimes \shift_m C_2$. We do this by first constructing an almost chain map between the desired complexes, and then applying Lemma~\ref{lem:almosttolocal}. The construction of $f_{U, m}$ (and the verification that it is an almost chain map) will be the most technical part of the argument and will occupy our attention for the next few pages.

\begin{definition}
Let $\{ w, y_i, z_i \}$ and $\{ w', y'_i, z'_i \}$ be the $U$-simplified bases for $C_3$ and $\shift_{U,m}(C_3)$, respectively. Define
\[ f_{U,m} \co \shift_{U,m}(C_3) \to \shift_{U,m}(C_1) \otimes \shift_{U,m}(C_2) \]
by first setting
\[ f_{U,m} (r') = \sigma_{U,m} f(r) \]
whenever $r' \in \{ w', z'_i\}$. To define $f_{U,m}(y'_i)$, we proceed with some casework. Write $f(y_i)$ in terms of the $U$-simplified basis for $C_1 \otimes C_2$, so that
\[ f(y_i) = \sum_{j \in J_1} \kappa_j + \sum_{j \in J_2} U^{p_j} \kappa_j + \sum_{j \in J_3} V^{q_j} \kappa_j + \sum_j P_j(U,V)\lambda_j + Q(U,V) \omega, \]
for some $p_j, q_j \in \N$, $P_j, Q \in \cR$, and disjoint index sets $J_1, J_2$, and $J_3$. We define $f_{U,m} (y'_i)$ based on the value of $|c_{2i-1}|$. If $|c_{2i-1}| < m$, let
\begin{align*} 
	f_{U,m} (y'_i) &= \sigma_{U,m} f(y_i) \\
				&=  \sigma_{U,m} \Big( \sum_{j \in J_1} \kappa_j + \sum_{j \in J_2} U^{p_j} \kappa_j + \sum_{j \in J_3} V^{q_j} \kappa_j +  \\
				&\qquad \qquad \sum_j P_j(U,V)\lambda_j + Q(U,V) \omega \Big),
\end{align*}
as before. If $|c_{2i-1}| \geq m$, let
\begin{align*} 
	f_{U,m} (y'_i) &=  \sigma_{U,m} \Big( \sum_{j \in J_1} U^{\overline{\tau}(j)} \kappa_j + \sum_{j \in J_2} U^{p_j+\overline{\tau}(j)} \kappa_j + \sum_{j \in J_3} V^{q_j} \kappa_j + \\
	&\qquad \qquad \sum_j P_j(U,V)\lambda_j + Q(U,V) \omega \Big),
\end{align*}
where 
\[ \overline{\tau}(j) = \begin{cases}
			1 &\text{ if } e_j < m \\
			0 &\text{ if } e_j \geq m.	
		\end{cases}
\]
Observe that $\tau(j) + \overline{\tau}(j) = 1$. In addition, note that if $f(y_i)$ is supported by $\kappa_j$, then $e_j \geq |c_{2i-1}|$. This follows from the fact that $\partial_U f(y_i) = f(\partial y_i)$ is in $\im U^{|c_{2i-1}|}$, while $\partial_U \kappa_j = U^{e_j} \lambda_j$. Hence in particular if $|c_{2i-1}| \geq m$, then for for any $j \in J_1$, we must have $\overline{\tau}(j) = 0$. (Thus we could have omitted the very first instance of $U^{\overline{\tau}(j)}$ in the above definition of $f_{U,m} (y'_i)$, but we have left it in for future notational convenience.)

We also note that 
\[ f(U^{|c_{2i-1}|} z_i) = f \d_U(y_i) = \d_U f(y_i) = \sum_{j \in J_1}U^{e_j} \lambda_j +  \sum_{j \in J_2} U^{p_j+e_j} \lambda_j, \]
hence 
\begin{equation}\label{eq:sumfz}
	U^{|c_{2i-1}|} \sigma_{U,m} f( z_i) = \sum_{j \in J_1}U^{e_j}  \sigma_{U,m} (\lambda_j) +  \sum_{j \in J_2} U^{p_j+e_j}  \sigma_{U,m} (\lambda_j).
\end{equation}
Finally, note that
\begin{equation}\label{eq:sigmaf}
	\sigma_{U,m} f(r) \equiv f_{U,m} s_{U,m}(r) \mod U
\end{equation}
for all $r \in \{ w, y_i, z_i \}$. Indeed, if $r = w$ or $z_i$, this congruence is an equality by definition; whereas if $r = y_i$, then the claim follows from the fact that (in the $|c_{2i-1}| \geq m$ case) $\overline{\tau}(j) = 0$ for all $j \in J_1$.
\end{definition}

\begin{lemma}\label{lem:fUalmost}
Let $f \co C_3 \to C_1 \otimes C_2$ be a local map. Then $f_{U,m}$ is an almost chain map.
\end{lemma}

\begin{proof}
Let $\{ w, y_i, z_i \}$ be the $U$-simplified basis for $C_3 = C(c_1, \dots, c_{n_3})$.
It suffices to show
\begin{equation}\label{eq:Ualmostchain}
			\d_U f_{U,m} (y'_i) \equiv \begin{cases}
				U^{|c_{2i-1}|}  \sigma_{U,m} f(z_i) \mod U^{|c_{2i-1}|+1} &\text{ if } |c_{2i-1}| < m \\
				U^{|c_{2i-1}|+1}  \sigma_{U,m} f(z_i) \mod U^{|c_{2i-1}|+2} &\text{ if } |c_{2i-1}| \geq m,
			\end{cases}
\end{equation}
and that 
\begin{equation}\label{eq:Valmostchain}
	\d_V f_{U,m}(r') \equiv f_{U,m} \d_V (r')\mod U
\end{equation}
for all $r' \in \{ w', y'_i, z'_i \}$.
(The mod $U$ in the above equation is not necessary, since our complexes are reduced by assumption, but is included for emphasis.)

We first consider \eqref{eq:Ualmostchain}. Suppose $|c_{2i-1}| < m$. Then 
\begin{align*}
	\d_U f_{U,m} (y'_i) &=  \d_U \Big(   \sum_{j \in J_1} \sigma_{U,m}(\kappa_j) + \sum_{j \in J_2} U^{p_j}\sigma_{U,m}(\kappa_j) \Big) \\
			&=  \sum_{j \in J_1} U^{e_j+\tau(j)} \lambda'_j+  \sum_{j \in J_2}U^{p_j} U^{e_j+\tau(j)} \lambda'_j \\
			&=  \sum_{j \in J_1} U^{e_j+\tau(j)} \sigma_{U,m} (\lambda_j) + \sum_{j \in J_2}U^{p_j} U^{e_j+\tau(j)} \sigma_{U,m}(\lambda_j) \\ 
			&\equiv U^{|c_{2i-1}|}  \sigma_{U,m} f(z_i) \mod U^{|c_{2i-1}|+1}.
\end{align*}
Here, to obtain the last line, we compare the third line with (\ref{eq:sumfz}), and use the fact that if $\tau(j) = 1$, then $e_j \geq m > |c_{2i-1}|$.

Now suppose $|c_{2i-1}| \geq m$. We have
\begin{align*}
	\d_U f_{U,m} (y'_i) &=  \d_U  \Big(   \sum_{j \in J_1} U^{\overline{\tau}(j)} \sigma_{U,m}(\kappa_j) + \sum_{j \in J_2}U^{p_j+\overline{\tau}(j)}\sigma_{U,m}(\kappa_j) \Big) \\
			&= \sum_{j \in J_1}  U^{e_j+1} \lambda'_j+ \sum_{j \in J_2}U^{p_j+\overline{\tau}(j)} U^{e_j+\tau(j)} \lambda'_j \\
			&= \sum_{j \in J_1}  U^{e_j+1} \sigma_{U,m} (\lambda_j) + \sum_{j \in J_2}U^{p_j+e_j+1} \sigma_{U,m} (\lambda_j) \\
			&= U^{|c_{2i-1}|+1} \sigma_{U,m} f(z_i).			
\end{align*}
where in the last line we have used \eqref{eq:sumfz}.

We now consider \eqref{eq:Valmostchain}. We have
\begin{align*}
	\d_V f_{U,m}(r')  &\equiv \d_V \sigma_{U,m} f(r) \mod U \\
			&\equiv \sigma_{U,m} f \d_V (r) \mod U \\
			&\equiv  f_{U,m} s_{U,m} \d_V (r) \mod U \\
			&\equiv f_{U,m} \d_V (s_{U,m} (r)) \mod U \\
			&\equiv f_{U,m} \d_V (r') \mod U
\end{align*}
for any $r' \in \{ w', y'_i, z'_i \}$ (in fact, for any $r \in C_3$), where the first equivalence is by definition, the second since $\d_V$ commutes with $\sigma_{U,m}$ and $f$, the third by \eqref{eq:sigmaf}, and the fourth since $\d_V$ and $s_{U,m}$ commute.
\end{proof}

We now verify the remaining hypotheses of Lemma~\ref{lem:almosttolocal}. In the proofs of the following lemmas, we denote the standard preferred basis for $\shift_m(C_3)$ by $\{ x'_i \}$, and the $U$-simplified basis by $\{w', y'_j, z'_j\}$ as usual. 

\begin{lemma}\label{lem:fUtower}
With the notation as above, $[f_{U,m}(x'_0)]_{u'_0, v'_0}$ represents a $V$-tower class.
\end{lemma}

\begin{proof}
Note that $x'_0$ is one of $w'$, $y'_j$, or $z'_j$ for some $j$. If $x'_0 = w'$ or $z'_j$, then $f_{U,m}(x'_0) = \sigma_{U,m}f(x_0)$. The result now follows from the fact that $f$ is local and $\sigma_{U,m}$ induces an ungraded isomorphism between $(C_1\otimes C_2) / U$ and $(\shift_m(C_1) \otimes \shift_m(C_2))/U$. If $x'_0 = y'_j$, then $f_{U,m}(x'_0) \equiv \sigma_{U,m} f(x_0) \mod U$, and the result follows as before.
\end{proof}

\begin{lemma}\label{lem:fUdUxn}
With the notation as above, $\d_U[f_{U,m}(x'_n)]_{u'_n, v'_n} = 0.$ 
\end{lemma}

\begin{proof}
Recall that $x'_n = w'$. Therefore, we have $f_{U,m}(x'_n) = \sigma_{U,m}f(x_n)$. Since $f$ is an $\cR$-equivariant chain map and $x_n$ is a $U$-cycle, it follows that $f(x_n)$ is also a $U$-cycle. An examination of the definition shows that $\sigma_{U,m}$ takes $U$-cyles to $U$-cycles, so $\d_U f_{U,m}(x'_n) = 0.$ 
\end{proof}

\noindent
Putting everything together, we have:

\begin{lemma}\label{lem:Ulocal}
Let $f \co C_3 \to C_1 \otimes C_2$ be a  local map. Then there exists a local map
\[ g \co \shift_{U,m} (C_3) \to \shift_{U,m} (C_1) \otimes \shift_{U,m} (C_2). \]
\end{lemma}

\begin{proof}
By Lemma \ref{lem:fUalmost}, the map $f_{U,m}$ is an almost chain map; by Lemma \ref{lem:fUtower}, $[f_{U,m}(x'_0)]_{u'_0, v'_0}$ represents a $V$-tower class; and by Lemma \ref{lem:fUdUxn}, $\d_U[f_{U,m}(x'_n)]_{u'_n, v'_n} = 0$. Thus Lemma \ref{lem:almosttolocal} gives us the desired local map.
\end{proof}

By reversing the roles of $U$ and $V$, we may similarly define $f_{V, m}$. We record the analogous set of lemmas below:

\begin{lemma}\label{lem:fValmost}
Let $f \co C_3 \to C_1 \otimes C_2$ be a  local map. With the notation as above, $f_{V,m}$ is an almost chain map.
\end{lemma}

\begin{proof}
The proof is identical to the proof of Lemma \ref{lem:fUalmost} after reversing the roles of $U$ and $V$.
\end{proof}

\begin{lemma}\label{lem:fVtower}
With the notation as above, $[f_{V,m}(x'_0)]_{u'_0, v'_0}$ represents a $V$-tower class.
\end{lemma}

\begin{proof}
By definition, $f_{V,m}(x'_0) = \sigma_{V,m}f(x_0)$. Since $f$ is local, $f(x_0)$ represents a $V$-tower class, and it is easy to check that $\sigma_{V,m}$ takes $V$-tower classes to $V$-tower classes.
\end{proof}

\begin{lemma}\label{lem:fVdVxn}
With the notation as above, $\d_U[f_{V,m}(x'_n)]_{u'_n, v'_n} = 0$.
\end{lemma}

\begin{proof}
We have
\[\d_U f_{V,m}(x'_n)= f_{V,m}\d_U(x'_n) = 0\]
where the first equality follows by the analogue of equation~\eqref{eq:Valmostchain}.
\end{proof}

\begin{lemma}\label{lem:Vlocal}
Let $f \co C_3 \to C_1 \otimes C_2$ be a  local map. Then there exists a local map
\[ g \co \shift_{V,m} (C_3) \to \shift_{V,m} (C_1) \otimes \shift_{V,m} (C_2). \]
\end{lemma}

\begin{proof}
By Lemma \ref{lem:fValmost} and, the map $f_{V,m}$ is an almost chain map; by Lemma \ref{lem:fVtower}, $[f_{V,m}(x'_0)]_{u'_0, v'_0}$ represents a $V$-tower class; and by Lemma \ref{lem:fVdVxn},  $\d_U[f_{V,m}(x'_n)]_{u'_n, v'_n} = 0$. Thus Lemma \ref{lem:almosttolocal} gives the desired local map.
\end{proof}

We now finally turn to the proof of Theorem \ref{thm:shift}:

\begin{proof}[Proof of Theorem \ref{thm:shift}]
Suppose that $C_3 \sim C_1 \otimes C_2$. Let $f \co C_3 \to  C_1 \otimes C_2$ be a  local map. By Lemma \ref{lem:Ulocal}, we have a local map
\[ g \co \shift_{U,m} (C_3) \to \shift_{U,m} (C_1) \otimes \shift_{U,m} (C_2), \]
that is, 
\begin{equation}\label{eq:shiftineq}
	\shift_{U,m} (C_3) \leq \shift_{U,m} (C_1) \otimes \shift_{U,m} (C_2).
\end{equation}
Dually, we have $C^\vee_3 \sim C^\vee_1 \otimes C^\vee_2$, and by the same argument
\begin{equation}\label{eq:shiftineqdual}
	\shift_{U,m} (C^\vee_3) \leq \shift_{U,m} (C^\vee_1) \otimes \shift_{U,m} (C^\vee_2). 
\end{equation}
Dualizing \eqref{eq:shiftineqdual}, applying Lemma \ref{lem:shiftdual}, and combining with \eqref{eq:shiftineq}, we obtain
\[ \shift_{U,m} (C_1) \otimes \shift_{U,m} (C_2) \leq \shift_{U,m} (C_3) \leq \shift_{U,m} (C_1) \otimes \shift_{U,m} (C_2). \]
Thus we have
\[ \shift_{U,m} (C_3) \sim \shift_{U,m} (C_1) \otimes \shift_{U,m} (C_2). \]
The analogous argument replacing $U$ with $V$ (using Lemma \ref{lem:Vlocal} instead of Lemma \ref{lem:Ulocal}) shows that 
\[ \shift_{V,m} (\shift_{U,m} (C_3)) \sim \shift_{V,m} (\shift_{U,m} (C_1)) \otimes \shift_{V,m} ( \shift_{U,m} (C_2)). \]
Since $\shift_{V,m} \circ \shift_{U,m} = \shift_m$, it follows that
\[ \shift_m(C_1 \otimes C_2) \sim \shift_m(C_1) \otimes \shift_m(C_2), \]
as desired.
\end{proof}

\subsection{Proof of Theorem \ref{thm:homs}}
We now turn to the proof that the $\varphi_j$ are additive. Note that by considering the composition
\[ P \circ \shift_m \co \KL \to 2\Z, \quad m \in \N \]
we obtain infinitely many homomorphisms from $\KL$ to $2\Z$. The proof of Theorem \ref{thm:homs} relies on considering certain linear combinations of these homomorphisms.

\begin{proof}[Proof of Theorem \ref{thm:homs}]
Let $C \in \KL$. Since all of our maps are local equivalence invariants, we may assume that $C=C(a_1, \dots, a_n)$ is a standard complex. For any $m \in \N$, write
\[ P(\shift_m(C)) = -2 \sum_{1 \leq j < m} j \varphi_j(C) -2 \sum_{j \geq m} (j+1) \varphi_j(C) + \sum_{i=0}^n \sgn a_i. \]
Here, we have simply used the definition of $\shift_m$, together with the definition of $\varphi_j$ as a count of standard complex parameters. This implies that
\begin{equation}\label{eq:Pshift}
P(\shift_m(C)) - P(C) = -2 \sum_{j \geq m} \varphi_j(C).
\end{equation}

We now use (strong, downward) induction to show that $\varphi_j$ is a homomorphism for all $j \in \N$. Fix $C_1, C_2 \in \KL$, where $C_1 = C(a_1, \dots, a_{n_1})$ and $C_2 = C(b_1, \dots, b_{n_2})$. For
\[ N > \max \{ a_i, b_j \}, \]
we have $\varphi_N(C_1) = \varphi_N(C_2) = \varphi_N(C_1 \otimes C_2) = 0$. This establishes the base case. Thus, assume that $\varphi_j$ is a homomorphism for all $j \geq M+1$. We will show that $\varphi_{M}$ is also a homomorphism. Indeed, 
\begin{align*}
-2 \sum_{j \geq M} \Big( \varphi_j(C_1) + \varphi_j(C_2) \Big) &= P(\shift_{M}(C_1)) + P(\shift_{M}(C_2)) - P(C_1) -P(C_2)\\[-0.7em]
	&=P(\shift_{M}(C_1 \otimes C_2)) - P(C_1 \otimes C_2) \\
	&= -2 \sum_{j \geq M} \varphi_j(C_1 \otimes C_2),
\end{align*}
where the first and third equalities follow from \eqref{eq:Pshift}, and the second equality follows from the fact that $P$ and $\shift_M$ are homomorphisms. By the inductive hypothesis, we have that $\varphi_j$ is a homomorphism for all $j \geq M+1$. It follows that $\varphi_M$ is a homomorphism as well. This completes the proof. 
\end{proof}

\subsection{$\HFK^-$ and $\varphi_j$}
We are now ready to prove Proposition \ref{prop:HFKm}. Recall that \[N(K) = \begin{cases}
0 & \text{if } \varphi_j(K)=0 \text{ for all } j,
\\
\max \{ j \mid \varphi_j(K) \neq 0\} & otherwise.
\end{cases}\]

\propHFKm*
\begin{proof}
Let $C = C(a_1, \dots, a_n)$ be the standard complex representative of $\HFK^-(K)$ given by Theorem \ref{thm:char} and Corollary \ref{cor:splitting}. Recall that $H_*(\CFKUV(K)/V) \cong \HFK^-(K)$. Then $U^M \cdot \Tors_U  HFK^-(K) = 0$ implies $U^M \cdot \Tors_U  H_*(C/V)= 0$, which in turn implies that $a_i \leq M$ for $i$ odd. The result now follows from the definition of $\varphi_j$.
\end{proof}

\section{Thin knots and L-space knots}\label{sec:thinLspace}
In this section, we prove Propositions \ref{prop:thin} and \ref{prop:Lspace}.

\propthin*

\begin{proof}
By \cite[Theorem 4]{Petkovacables}, it follows that if $K$ is a thin knot, then $\CFKUV(K)$ is locally equivalent to the standard complex $C(a_1, \dots, a_n)$ where $n = 2|\tau(K)|$ and $a_i = \sgn \tau(K)$ for $i$ odd and $a_i = -\sgn \tau(K)$ for $i$ even. That is, the $a_i$ are an alternating sequence of $\pm 1$, starting with $+1$ if $\tau(K) > 0$ and $-1$ if $\tau(K)<0$. The result follows.
\end{proof}

\propLspace*

\begin{proof}
By \cite[Theorem 1.2]{OSlens} (cf. \cite[Theorem 2.10]{OSS}), we have that if $K$ is an L-space knot, then $\CFKUV(K)$ is the standard complex 
\[ C(c_1, -c_n, c_2, -c_{n-1}, c_3, -c_{n-2}, \dots , c_n, -c_1).\]
The result now follows from the definition of $\varphi_j$.
\end{proof}

\section{An infinite-rank summand of topologically slice knots}\label{sec:topslice}
The goal of this section is to prove Theorem \ref{thm:TS}. Let $D$ be the (untwisted, positively-clasped) Whitehead double of the right-handed trefoil. 
Consider $K_n = D_{n, n+1} \# - T_{n, n+1}$. The knots $K_n$ are topologically slice and will generate a $\mathbb{Z}^\infty$-summand of $\mathcal{C}_{TS}$. Indeed, the knot $D$ has Alexander polynomial one, and hence is topologically slice. Thus, the cable $D_{n,n+1}$ is topologically concordant to the underlying pattern torus knot $T_{n,n+1}$, and so $D_{n,n+1}\# -T_{n,n+1}$ is topologically slice.

\begin{proposition}\label{prop:cablesofD}
Let $D_{n, n+1}$ denote the $(n, n+1)$ cable of the (untwisted, positively-clasped)  Whitehead double of the right-handed trefoil. Then
 \begin{eqnarray*}
\varphi_j(D_{n, n+1}) = \begin{cases}
n, & j = 1;
\\
1, & 1<  j < (n-1) \text{ or } j = n;
\\
0, & j = n-1 > 1 \text{ or }  j > n.
\end{cases}
\end{eqnarray*}
\end{proposition}
\begin{proof}
By Lemma 6.12 of \cite{Homconcordance}, the knot $D$ is $\varepsilon$-equivalent to $T_{2,3}$. Thus, by Proposition 4 of \cite{Homconcordance}, we may consider $\CFKUV(T_{2,3;n,n+1})$, where $T_{2,3;n,n+1}$ denotes
the $(n, n + 1)$-cable of $T_{2,3}$, instead of the locally equivalent $\CFKUV(D_{n,n+1})$. The advantage of this approach is that $T_{2,3;n,n+1}$ is an L-space knot
\cite[Theorem 1.10]{Hedden2} (cf. \cite{HomcablesLspace}), and so $\CFKUV(T_{2,3;n,n+1})$ is a standard complex and completely determined by its Alexander polynomial \cite[Theorem 1.2]{OSlens}.

It follows from \cite[Lemma 6.7]{Homconcordance} (also see the proof of \cite[Proposition 6.1]{heddenlivingstonruberman}) that
\begin{eqnarray*}
\Delta_{T_{n, n+1}}(t) = \sum_{i=0}^{n-1} t^{in} - t \sum_{i=0}^{n-2} t^{i(n+1)}.
\end{eqnarray*}
Recall that the Alexander polynomial of a cable knot is determined by
\begin{eqnarray*}
\Delta_{K_{p, q}}(t) = \Delta_K(t^p) \cdot \Delta_{T_{p, q}}(t).
\end{eqnarray*}
This gives
\begin{eqnarray*}
\Delta_{T_{2,3;n,n+1}}(t) &=& \Delta_{T_{2,3}}(t^n) \cdot \Delta_{T_{n, n+1}}(t)
\\
&=&( t^{2n} - t^n +1)\cdot \left( \sum_{i=0}^{n-1} t^{in} - t \sum_{i=0}^{n-2} t^{i(n+1)}\right).
\end{eqnarray*}
For small values of $n$, we have:
\begin{eqnarray*}
\Delta_{T_{2,3;2,3}}(t) &=&  t^6-t^5+t^3-t+1
\\
\Delta_{T_{2,3;3,4}}(t) &=&  t^{12} -t^{11}+t^8-t^7+t^6-t^5+t^4-t+1.
\end{eqnarray*}
For $n \geq 4$, we rearrange and simplify as follows. We first observe the following telescoping sum
\begin{align*} 
	(- t^n +1)\cdot \left( \sum_{i=0}^{n-1} t^{in}\right) &= 1-t^{n^2}.
\end{align*}
We also have
\[ (- t^n +1)\cdot  \left(-t \sum_{i=0}^{n-2} t^{i(n+1)}\right) =  - t + \sum_{j = 1}^{n-2}\Big( t^{j(n+1)} - t^{j(n+1)+1}\Big) + t^{(n-1)(n+1)} \]
and
\[( t^{2n})\cdot \left( \sum_{i=0}^{n-1} t^{in} - t \sum_{i=0}^{n-2} t^{i(n+1)}\right) = \sum_{i=2}^{n+1} t^{in} -  \sum_{j=2}^{n} t^{j(n+1)-1}.\]
Putting the two simplifications together, we get:
\begin{eqnarray*}
\Delta_{T_{2,3;n,n+1}}(t) &=& 1 - t + t^{n+1} - t^{n+2} + \sum_{j = 2}^{n-2}\Big( -t^{j(n+1)-1} + t^{j(n+1)} - t^{j(n+1)+1}\Big) 
\\
& &   +  \sum_{i=2}^{n+1} t^{in} - t^{n^2} -  t^{(n-1)(n+1)-1} -  t^{n(n+1)-1}+ t^{(n-1)(n+1)} 
\\
&=& 1 - t + t^{n+1} - t^{n+2} + \sum_{j = 2}^{n-2}\Big( t^{jn} -t^{jn+j-1} + t^{jn+j} - t^{jn+j+1}\Big) 
\\
& &   + t^{n^2-n} -  t^{n^2-2}+ t^{n^2-1}  -  t^{n^2+n-1}+ t^{n^2+n}.
\end{eqnarray*}
In particular, the number of terms in the Alexander polynomial is $4\cdot(n-1)+1$.

Thus, we have
\[\Delta_{T_{2,3;n,n+1}}(t) = \sum_{i=0}^{4(n-1)} (-1)^{i}t^{b_i}\]
where $(b_i)_{i=0}^{4(n-1)}$ is the decreasing sequence of integers found above. Defining 
\[c_i = b_{2i-2}-b_{2i-1}, \quad 1 \leq i \leq 2(n-1),\]
one readily checks that for $1 \leq i \leq 2(n-1)$, 
\begin{eqnarray*}
c_i (T_{2,3;n,n+1})= \begin{cases}
(i-1)/2, & i \text{ is odd}, i > 1;
\\
n, & i = 2(n-1);
\\
1, & \text{otherwise.}
\end{cases}
\end{eqnarray*}
Since $T_{2,3; n, n+1}$ is an L-space knot, by Proposition \ref{prop:Lspace} we have $\varphi_j(T_{2,3;n,n+1}) = \# \{ c_i \ | \ c_i = j\} $, and the calculation of $\varphi_j(D_{n,n+1})$ (which  equals $\varphi_j(T_{2,3;n,n+1})$) follows immediately.  \end{proof}

We now prove Theorem~\ref{thm:TS} to produce an infinite rank summand of $\CTS$. 
\begin{proof}[Proof of Theorem~\ref{thm:TS}]
Recall Example~\ref{ex:torusknots}, which states that the torus knot $T_{n, n+1}$ has
\[ \varphi_j(T_{n, n+1}) = \begin{cases}
	1, &\text{ if } j=1,2, \dots, n-1; \\
	0, &\text{ otherwise.}
\end{cases}
\]
By Proposition~\ref{prop:cablesofD} and the fact that $\varphi_j$ is a homomorphism (Theorem~\ref{thm:homs}), we have that \[\varphi_n(K_n) = \varphi_n(D_{n,n+1}) - \varphi_n(T_{n,n+1}) = 1\] and $\varphi_j(K_n) = 0$ if $j > n$. The theorem now follows from a straightforward linear algebra argument; see, for example, \cite[Lemma 6.4]{OSS}.
\end{proof}

\section{Concordance genus and concordance unknotting number}\label{sec:gcuc}
In this section, we discuss applications of our homomorphisms to concordance genus and concordance unknotting number.
\subsection{Concordance genus}
Recall that knot Floer homology detects genus \cite{OSgenus}. Using the conventions and notation from Section \ref{background}, we have that
\[ g(K) = \frac{1}{2} \max \{ A([x]) -A([y])  \mid [x], [y] \neq 0 \in H_*(\CFKUV(K)/(U, V)) \}. \]

\begin{proof}[Proof of Theorem \ref{thm:gcuc} \eqref{it:gc}]
Suppose that $K'$ is concordant to $K$. Let $N = N(K) = N(K') = \max \{ j \mid \varphi_j(K) \neq 0\}$.  By Theorem \ref{thm:char} and Corollary \ref{cor:splitting}, we have that there exist $[x], [y] \neq 0\in H_*(\CFKUV(K')/(U,V))$ with $\gr(x) - \gr(y) = (1, 1-2N)$ or $(1-2N, 1)$, depending on the sign of $\varphi_N(K)$. In either case
\[ |A(x) - A(y)| = N, \]
implying that $g(K') \geq N/2$. Thus, $g_c(K) \geq N/2$, as desired.
\end{proof}

\subsection{Concordance unknotting number}
We recall the following definitions and results from \cite{AlishahiEftekharyunknotting}. (The results are originally stated over the ring $\F[U, V]$; quotienting by $UV$ yields the results as stated here.)

Let $u'(K)$ be the least integer $m$ such that there exist grading-homogenous $\cR$-equivariant chain maps
\[ f \co \CFKUV(K) \to \cR \quad \text{ and } \quad g \co \cR \to \CFKUV(K) \]
such that $g \circ f$ is homotopic to multiplication by $U^m$ and $f \circ g$ is multiplication by $U^m$.

\begin{theorem}[{\cite[Theorem 1.1]{AlishahiEftekharyunknotting}}]
The integer $u'(K)$ is a lower bound for the unknotting number $u(K)$.
\end{theorem}

\begin{proof}[Proof of Theorem \ref{thm:gcuc} \eqref{it:uc}]
Suppose that $K'$ is concordant to $K$. Let  $N=N(K) = N(K') = \max \{ j \mid \varphi_j(K') \neq 0\}$. Note that this implies \[U^{N-1} \Tors_U H_*(\CFKUV(K')/V) \neq 0,\] where $\Tors_U M$ denotes the $U$-torsion submodule of an $\F[U]$-module $M$.

Let $u' = u'(K')$. Then there exist grading-homogenous $\cR$-equivariant chain maps
\[ f \co \CFKUV(K') \to \cR \quad \text{ and } \quad g \co \cR \to \CFKUV(K') \]
such that $g \circ f$ is homotopic to multiplication by $U^{u'}$ and $f \circ g$ is multiplication by $U^{u'}$. Now quotient by $V$. Since $g \circ f$ factors through $\cR$, it follows that $U^{u'}$ must annihilate $\Tors_U H_*(\CFKUV(K')/V)$, i.e., $u' \geq N$. This implies that $u_c(K) \geq N$, as desired.
\end{proof}

\begin{proof}[Proof of Theorem \ref{thm:topsliceuc}]
Let $K_n$ denote $D_{n, 1} \# -D_{n-1, 1}$ for $n \in \N$, where, as above, $D$ denotes the positively-clasped, untwisted Whitehead double of the right-handed trefoil. The knots $K_n$ are topologically slice, since $D_{m,1}$ is. These knots are used in \cite[Theorem 3]{Homconcordancegenus}. In particular, by \cite[Lemma 3.1]{Homconcordancegenus}, we have that $g_4(K_n) = 1$ for all $n$. By \cite[Lemma 3.3]{Homconcordancegenus}, we have that $a_1(D_{n, 1}) = 1$ and $a_2(D_{n, 1}) = -n$. (There is a difference in sign conventions between $a_2$ in \cite{Homconcordancegenus} and the present paper.) By \cite[Lemma 3.2]{Homconcordancegenus}, we have that $|a_{2i}(D_{n,1})| \leq n$ for all $i$, with equality if and only if $i=1$ by \cite[Lemma 3.3]{Homconcordancegenus}. It follows that $\varphi_n(D_{n,1}) = 1$ and $\varphi_i(D_{n,1}) = 0$ for all $i > n$. Hence $N(K_n) = n$, and by Theorem \ref{thm:gcuc} \eqref{it:uc}, we have that $u_c(K_n) \geq n$.
\end{proof}

\section{Further Remarks}\label{sec:furtherremarks}
We conclude with some remarks on knot-like complexes.

\subsection{Realizability}
The question of which knot-like complexes can be realized by knots in $S^3$ is difficult. See \cite{HeddenWatson} and \cite{KrcatovichLspace} for some restrictions. Note that their restrictions apply to the homotopy type, rather than local equivalence type, of knot-like complexes. For example, the standard complex $C(2, -2)$ is not realizable \cite[Theorem 7]{HeddenWatson} up to homotopy, but is realizable up to local equivalence \cite[Lemma 2.1]{HomUpsilon}.

Instead, we turn to the following purely algebraic question.
\begin{question} 
Which knot-like complexes are the mod $UV$ reduction of chain complexes over $\F[U,V]$?
\end{question}
\noindent Indeed, in Section \ref{background}, we defined the complex $\CFKUV(K)$ over the ring $\cR$, but the definition works equally well over $\F[U, V]$. Thus, in order for a knot-like complex $C$ to be realizable as coming from a knot $K \subset S^3$ up to homotopy (resp. local) equivalence, it is necessary for $C$ to be homotopy (resp. locally) equivalent to a complex that is the mod $UV$ reduction of a complex over $\F[U, V]$.

Na\"ively, one may hope to ``undo'' modding out by $UV$. That is, given a standard complex $C(a_1, \dots, a_n) = ( \cR\langle x_i \rangle, \d)$, one may hope to define a chain complex over $\F[U,V]$ by $C' = (\F[U, V]\langle x_i \rangle, \d')$, where $\d'$ is obtained by extending $\d$ linearly with respect to $\F[U, V]$. However, in general, $\d'^2$ will not be zero. As the following examples show, in some cases, the failure of $\d'^2 = 0$ can be remedied, while in other cases, it is fatal.

\begin{example}
We apply the above procedure to the standard complex 
\[ C(1, -2, -1, 1, 2, -1)\] 
from Example \ref{ex:121121}. Let $C'$ be generated over $\F[U, V]$ by 
\[ x_0, x_1, x_2, x_3, x_4, x_5, x_6 \]
with nonzero differentials
\begin{align*}
	\d' x_1 &= Ux_0 + V^2 x_2 \\
	\d' x_2 &= Ux_3 \\
	\d' x_4 &= V x_3 \\
	\d' x_5 &= U^2 x_4 + V x_6.
\end{align*}
Then $\d'^2 x_1 = UV^2 x_3 \neq 0$ and $\d'^2 x_5 = U^2V x_3 \neq 0$. However, if we instead endow $C'$ with the differentials
\begin{align*}
 \d' x_1 &= Ux_0 + V^2 x_2 + UVx_4\\ 
 \d' x_2 &= Ux_3, \\ 
 \d' x_4 &= V x_3 \\
  \d' x_5 &= UVx_2 + U^2 x_4 + V x_6
 \end{align*}
then $C'$ becomes a chain complex, as desired. Note that this change to the differential is equivalent to adding diagonals arrow from $x_1$ to $x_4$ and from $x_5$ to $x_2$ in Figure \ref{fig:121121}.
\end{example}

\begin{example}
We attempt to apply the above procedure to the standard complex $C(1, 1)$, generated by $x_0, x_1$, and $x_2$ with
\[ \d x_0 = 0, \quad \d x_1 = Ux_0, \quad \d x_2 = Vx_1. \]
Then $\d' x_2 = UV x_0 \neq 0$ and there is no way to modify $\d'$ so that is squares to zero and reduces mod $UV$ to $\d$.
\end{example}

More generally, one can show that any standard complex beginning with the parameters $a_1 = 1$ and $a_2 > 0$ cannot be realized as the mod $UV$ reduction of a chain complex over $\F[U, V]$, even up to local equivalence.

\subsection{Group structure of $\KL$}
Theorem \ref{thm:char} gives us a complete description of $\KL$ as a set; namely, the elements of $\KL$ are in bijection with finite sequences of nonzero integers. A natural question is the following:
\begin{question}
Is there is an explicit description of the group structure on $\KL$?
\end{question} 

In many simple cases, the group operation in $\KL$ simply concatenates or merges the sequences associated to the standard representatives.

\begin{example}
It follows from \cite[Theorem 4]{Petkovacables} that  $C(1, -1) \otimes C(1, -1) \sim C(1, -1, 1, -1)$. More generally, 
\[ C(1, -1, 1, -1, \dots, 1, -1) \otimes C(1, -1) \sim C(1, -1, 1, -1, \dots, 1, -1) \]
where the length of the right-hand side is the sum of the lengths of the factors on the left-hand side.
\end{example}

\begin{example}
By \cite[Lemma 2.1]{HomUpsilon}, we have that $C(1, -3, 3, -1) \otimes C(2, -2) \sim C(1, -3, 2, -2, 3, -1)$. 
\end{example}
\noindent
However, in general, the group operation in $\KL$ is more complicated:

\begin{example}\label{ex:C22C11}
One can show that
\[ C(2, -2) \otimes C(1, -1) \sim C(1, -1, 2, 1, -1, -2, 1, -1). \]
\end{example}
\noindent
Note that despite the seemingly complicated product structure exhibited in Example \ref{ex:C22C11}, the standard complex representative of a product of two standard complexes is highly constrained by the fact that $\varphi_j$ is a homomorphism for each $j \in \N$.

\bibliographystyle{amsalpha}
\bibliography{bib}

\end{document}